%% file: nncg_paper.tex
\documentclass[11pt]{article}

\usepackage{amsmath, amssymb, amsthm}
\usepackage{booktabs}
\usepackage{microtype}
\usepackage{graphicx}
\usepackage{algorithm}
\usepackage{algpseudocode}
\usepackage{authblk}
\usepackage[margin=1in]{geometry}
\usepackage[colorlinks=true, linkcolor=blue, citecolor=blue, urlcolor=blue]{hyperref}
\usepackage{xcolor}

\newtheorem{theorem}{Theorem}[section]
\newtheorem{lemma}{Lemma}[section]
\newtheorem{proposition}{Proposition}[section]
\newtheorem{corollary}{Corollary}[section]

\title{Non-Negative Conjugate Gradients}

\author[1]{Thomas Schmelzer}
\author[2]{Martin Stoll}

\affil[1]{\small Jebel Quant Research, Abu Dhabi}
\affil[2]{\small Faculty of Mathematics, TU Chemnitz, Germany}

\date{\today}

\input{tables/nncg_defs}
\input{tables/nncg_bench_defs}

\input{tables/nncg_regu_defs}

\input{tables/nncg_deblur_defs}
\input{tables/nncg_hyperspectral_library_defs}
\input{tables/nncg_precond_defs}
\input{tables/nncg_frontier_defs}
\input{tables/nncg_certificate_defs}

\begin{document}

\maketitle

\input{sections/s0_abstract}

\input{sections/s1_introduction}
\input{sections/s2_problem}
\input{sections/s3_reduction}
\input{sections/s5_nonneg}
\input{sections/s6_conditioning}
\input{sections/s7_results}
\input{sections/s8_conclusions}
\section*{Acknowledgements}
During the preparation of this work, the authors used Claude in order to draft the initial structure of several sections and for creating suitable test-cases for the numerical experiments. After using this tool, the authors reviewed and edited the content as well as the codes needed and take full responsibility for the content of the publication.
\appendix
\input{sections/s9_appendix}

\bibliographystyle{siam}
\bibliography{../matrix_free/bib/refs}

\end{document}

%% file: tables/nncg_defs.tex
\newcommand{\nncgCGslope}{0.35}
\newcommand{\nncgCGrtwo}{0.995}

\newcommand{\nncgMaxOuter}{6}
\newcommand{\nncgMaxErr}{7e-10}
\newcommand{\nncgPGratio}{16}
\newcommand{\nncgRatioSlope}{0.38}
\newcommand{\nncgKappaMax}{10^{4}}
\newcommand{\nncgEqErr}{9e-09}

\newcommand{\nncgTrajTotal}{15}
\newcommand{\nncgRankOkErr}{1e-09}
\newcommand{\nncgRankConvLow}{4}
\newcommand{\nncgRankConvHigh}{0}
\newcommand{\nncgPcgKappa}{9e+04}
\newcommand{\nncgPcgCG}{2983}
\newcommand{\nncgPcgPCG}{326}
\newcommand{\nncgPcgRatio}{9}
\newcommand{\nncgPcgBase}{243}
\newcommand{\nncgFbSeeds}{60}
\newcommand{\nncgFbCycles}{10}

\newcommand{\nncgFbFired}{17}
\newcommand{\nncgFbMax}{22}

%% file: tables/nncg_bench_defs.tex
\newcommand{\nncgBenchASCGSens}{20.5}

\newcommand{\nncgBenchRatioN}{2000}
\newcommand{\nncgBenchLHvsASD}{48}
\newcommand{\nncgBenchIPvsASD}{21}

\newcommand{\nncgBenchShawN}{1024}
\newcommand{\nncgBenchShawLHvsASCG}{50}

%% file: tables/nncg_regu_defs.tex
\newcommand{\nncgReguN}{128}
\newcommand{\nncgReguNoise}{10^{-3}}
\newcommand{\nncgReguKappa}{10^{20}}
\newcommand{\nncgReguAlpha}{10^{-4}}
\newcommand{\nncgReguKappaReg}{10^{5}}
\newcommand{\nncgReguOuter}{3}
\newcommand{\nncgReguActive}{4}
\newcommand{\nncgReguSupp}{124}

\newcommand{\nncgReguErrLH}{2e-04}
\newcommand{\nncgPhilN}{512}
\newcommand{\nncgPhilActive}{273}
\newcommand{\nncgPhilTrueZeros}{256}

\newcommand{\nncgPhilErrLH}{1e-04}

%% file: tables/nncg_deblur_defs.tex
\newcommand{\nncgDeblurNside}{128}
\newcommand{\nncgDeblurDim}{16384}
\newcommand{\nncgDeblurOuter}{23}
\newcommand{\nncgDeblurInner}{4314}
\newcommand{\nncgDeblurActive}{8936}
\newcommand{\nncgDeblurDenseGB}{2.1}
\newcommand{\nncgDeblurPSNRblur}{17.8}
\newcommand{\nncgDeblurPSNR}{22.7}
\newcommand{\nncgDeblurMaxNside}{256}
\newcommand{\nncgDeblurMaxDim}{65536}
\newcommand{\nncgDeblurMaxTime}{2.3}
\newcommand{\nncgDeblurMaxMPRGPTime}{1.5}
\newcommand{\nncgDeblurMaxDenseGB}{34}
\newcommand{\nncgDeblurHeadNside}{64}
\newcommand{\nncgDeblurHeadDim}{4096}
\newcommand{\nncgDeblurHeadMFTime}{0.09}
\newcommand{\nncgDeblurHeadLHTime}{26.5}
\newcommand{\nncgDeblurHeadIPTime}{6.6}
\newcommand{\nncgDeblurHeadLHx}{285}
\newcommand{\nncgDeblurHeadIPx}{71}

%% file: tables/nncg_hyperspectral_library_defs.tex
\newcommand{\nncgHsLibPixels}{500}
\newcommand{\nncgHsLibBands}{188}
\newcommand{\nncgHsLibSignatures}{498}
\newcommand{\nncgHsLibRank}{188}
\newcommand{\nncgHsLibRidgeFrac}{1e-06}
\newcommand{\nncgHsLibOuterCold}{5802}
\newcommand{\nncgHsLibOuterWarm}{510}
\newcommand{\nncgHsLibFracSingle}{99.8}
\newcommand{\nncgHsLibSpeedupOuter}{11.4}
\newcommand{\nncgHsLibSpeedupWall}{672.3}
\newcommand{\nncgHsLibAgree}{9e-13}
\newcommand{\nncgHsLibSpeedupFista}{3.7}

\newcommand{\nncgHsLibSupportMean}{1.0}
\newcommand{\nncgHsLibSupportMax}{1}

%% file: tables/nncg_precond_defs.tex
\newcommand{\nncgPrecondN}{500}
\newcommand{\nncgPrecondRank}{4}
\newcommand{\nncgPrecondNSolves}{60}
\newcommand{\nncgPrecondCrossover}{2}
\newcommand{\nncgPrecondFinalSpeedup}{1.06}
\newcommand{\nncgPrecondNystromFinalMs}{7174.6}
\newcommand{\nncgPrecondGlobalFinalMs}{6773.8}
\newcommand{\nncgPrecondMaxError}{4.3e-11}

%% file: tables/nncg_frontier_defs.tex
\newcommand{\nncgFrontN}{494}
\newcommand{\nncgFrontK}{10}
\newcommand{\nncgFrontStart}{July 2021}
\newcommand{\nncgFrontEnd}{May 2026}
\newcommand{\nncgFrontSysFrac}{50}

\newcommand{\nncgFrontGammas}{60}
\newcommand{\nncgFrontHeldMin}{1}
\newcommand{\nncgFrontHeldMax}{60}
\newcommand{\nncgWarmColdOuter}{8.4}
\newcommand{\nncgWarmOuter}{1.6}
\newcommand{\nncgWarmStable}{31}
\newcommand{\nncgWarmStableTot}{59}

\newcommand{\nncgWarmX}{15}
\newcommand{\nncgWarmInnerX}{11.8}
\newcommand{\nncgWarmErr}{2e-10}
\newcommand{\nncgFrontMaxN}{20000}
\newcommand{\nncgFrontMaxGB}{3.2}
\newcommand{\nncgFrontMaxX}{72}

%% file: tables/nncg_certificate_defs.tex
\newcommand{\nncgCertN}{200}
\newcommand{\nncgCertSeeds}{5}
\newcommand{\nncgCertMinMargin}{10^{-8}}
\newcommand{\nncgCertASouter}{13}
\newcommand{\nncgCertASresid}{6.7\times10^{-14}}
\newcommand{\nncgCertMPerrHi}{37}
\newcommand{\nncgCertMPtolHi}{10^{-6}}

\newcommand{\nncgCertMPdualHi}{-1.1\times10^{-4}}
\newcommand{\nncgCertMPtolZero}{10^{-10}}
\newcommand{\nncgCertMPitZero}{294}

%% file: sections/s0_abstract.tex
\begin{abstract}

The conjugate gradient method (CG) solves a symmetric positive definite (SPD) system $Ax = b$, but on its own it does not respect $x \geq 0$. We develop \emph{non-negative conjugate gradients}, a solver for the bound-constrained quadratic
\[
  \min_{x}\; \tfrac{1}{2}\,x^\top A x - b^\top x
  \quad\text{subject to}\quad x \geq 0,\;\; Bx = c,
  \qquad A \succ 0,
\]
where the equality system $Bx = c$ is present only in the equality-augmented variant.
The solver wraps the CG method in a primal-dual active-set loop where each step solves the unconstrained system
over a set of \emph{free} variables, drops any variables that turn negative, re-admits any bound variable whose reduced gradient is negative, and restarts CG. These toggles are the
principal pivots of the linear complementarity problem $(A,\,{-b})$, and because $A$ is a $P$-matrix,
every pivot is well defined. For termination we use the block-principal-pivoting
construction of Júdice and Pires, a fast block-pivot path safeguarded by a least-index Bland fallback, which
reaches the unique global minimiser in finitely many outer steps on the strictly monotone problem,
even under degeneracy. Our contribution carries that guarantee to the inexact,
matrix-free regime the Krylov inner solve requires. Meaning that once the CG algorithm's residual is tied to the decision margin,
the approximate loop provably makes the same primal and dual sign decisions as the exact one, so finite termination transfers decision-for-decision. The guarantee attaches to the outer loop, so a direct factorisation may replace CG unchanged.

Each inner solve is matrix-free in the sense that when matrix vector products with $A = M^\top M$ are required, the operator $v \mapsto M^\top(M v)$
needs two products with $M$ and no $O(n^2)$ storage, and it retains CG's $O(\sqrt\kappa)$ convergence rate, where $\kappa = \kappa(A)$. In order to improve the convergence behvaihour we show that any regularisation or preconditioner, which compresses the spectrum lowers the iteration count. A ridge/Tikhonov split of the form
$A \mapsto (1-\alpha)A + \alpha R^\top R$ also secures the needed $P$-matrix property. A loop-free projected-gradient reference method, analysed alongside, forgoes the Krylov subspace and converges at the slower $O(\kappa)$ rate. On synthetic SPD problems the loop reaches the planted optimum in at most $\nncgMaxOuter{}$ outer steps, and with a direct free-set solve it
outpaces Lawson--Hanson and an interior-point solver by more than an order of magnitude on dense instances.
\end{abstract}

%% file: sections/s1_introduction.tex
\section{Introduction}

The conjugate gradient method (CG) of Hestenes and Stiefel~\cite{hestenes1952} is the
canonical Krylov solver for a symmetric positive definite (SPD) linear system $A x = b$;
see, e.g., Golub and Van Loan~\cite{golub2013}, Greenbaum~\cite{greenbaum1997}, or Trefethen
and Bau~\cite{trefethenbau1997}. It builds its iterates in the Krylov subspace
$\mathcal{K}_k(A, b)$ and converges at the $O(\sqrt\kappa)$ rate, where $\kappa = \kappa(A)$
is the spectral condition number (Proposition~\ref{prop:cg}). CG solves linear systems; it is \emph{not} designed for inequalities, that is, for constrained programs considered here. This paper applies CG within the context of solving the bound-constrained quadratic
\begin{equation}\label{eq:bqp}
  \min_{x \in \mathbb{R}^n}\; f(x) = \tfrac{1}{2}\,x^\top A x - b^\top x
  \quad\text{subject to}\quad x \geq 0,
  \qquad A \succ 0,
\end{equation}
without giving up its per-iteration cost or its Krylov convergence rate.

Problem~\eqref{eq:bqp} is the strictly convex non-negative quadratic program. When
$A = M^\top M$ is the Gram matrix of $M \in \mathbb{R}^{m\times n}$ and $b = M^\top d$, its minimiser solves the non-negative least-squares (NNLS) problem
$\min_{x\geq 0}\|Mx - d\|_2^2$, the setting consider in Lawson and Hanson~\cite{lawson1974}. Without the constraint, the minimiser is simply $x = A^{-1}b$, which is exactly what CG computes. The constraint $x \geq 0$ is what stands between that clean single solve and the answer to our problem. We also treat
the \emph{equality-augmented} variant, in which a linear equality system $Bx = c$ (with
$B \in \mathbb{R}^{p\times n}$ of full row rank) is imposed alongside $x \geq 0$. This variant
arises whenever the solution must satisfy a family of linear balance conditions, the single
normalisation $\mathbf{1}^\top x = \beta$ being the case $p=1$. Its multipliers
$\lambda \in \mathbb{R}^p$ can be eliminated analytically through a $p\times p$ Schur
complement (cf.\ Section~\ref{sec:reduction}).

This paper studies the constraint-handling layer directly, and its central contribution is a finite \emph{termination guarantee}. The method of our choice is a \emph{primal-dual active-set} outer
loop~\cite{nocedal2006} that wraps around CG. It solves the unconstrained system over a working set
of \emph{free} variables, releases any variable that returns negative (the primal step),
re-admits any bound variable whose reduced gradient is negative (the dual step), and restarts
CG on the modified free set each time. The variable toggles are precisely the \emph{principal
pivots} of the linear complementarity problem (LCP) $(A,\,{-b})$. Since $A \succ 0$ makes it a
$P$-matrix, every principal submatrix is invertible and each pivot is well
defined~\cite{murty1988,judice1994,kim2011}. Guarding a fast block-pivot path with a
least-index Bland-style fallback yields finite termination at the unique global minimiser
even under degeneracy, the block-principal-pivoting construction of Júdice and
Pires~\cite{judice1994} (Theorem~\ref{thm:termination}). What we add is its transfer to the
inexact CG solves the loop actually performs: a perturbation lemma shows that
once the residual is small against the problem's decision margin, the inexact loop visits
exactly the free sets of the exact one (Lemma~\ref{lem:inexact}). Combining this guarantee with a Krylov inner solver on a changing free set turns it into a fast solver in practice. In concrete terms, CG restarts on each reduced system and warm-starts across a parametric family of nearby problems.

The inner solve preserves two properties that matter for computational efficiency. First, it is
\emph{matrix-free}, i.e. when $A = M^\top M$ is applied, the action
$v \mapsto A_{\mathcal{F}}v = M_{\mathcal{F}}^\top(M_{\mathcal{F}}v)$ on the free set $\mathcal{F}$ is evaluated as two products with $M_{\mathcal{F}}$, so $A$ is \emph{never}
assembled and no $O(n^2)$ storage is required. Second, its iteration count is governed by the condition number $\kappa$, so any preconditioner that compresses the spectrum accelerates
convergence. A ridge/Tikhonov splitting $A \mapsto A_\alpha = (1-\alpha)A + \alpha R^\top R$
with $R^\top R \succ 0$ does both. It lowers $\kappa$ (Proposition~\ref{prop:regcond}), and,
by keeping every principal submatrix strictly positive definite, it secures the $P$-matrix
property that the finite-termination proof relies on (Section~\ref{sec:conditioning}).

For contrast we analyse a loop-free reference method that enforces the constraint at every
step. This \emph{projected-gradient} method takes a gradient step and projects onto the
feasible set. The feasible set is the non-negative orthant for~\eqref{eq:bqp}, or the simplex
slice $\Delta_\beta = \{x : \mathbf{1}^\top x = \beta,\, x \geq 0\}$ for the
single-normalisation case, projected onto in $O(n\log n)$ by a sort-and-threshold pass
(cf.\ \cite{duchi2008,condat2016}). The method is single-stage and simple. It is not a Krylov
subspace method, it does not orthogonalise residuals against prior search directions, and it converges at the $O(\kappa)$ rate, a factor $\sqrt\kappa$ slower than the CG inner loop (Section~\ref{sec:nonneg}).

Both building blocks are individually classical. One is an active-set method and a block-principal-pivoting method for the non-negative and complementarity
problems~\cite{lawson1974,judice1994,kim2011,murty1988,nocedal2006}. The other is CG with preconditioning for SPD systems~\cite{hestenes1952,golub2013,greenbaum1997,trefethenbau1997,%
benzi2005}. What is \emph{not} classical is carrying that termination guarantee to inexact, matrix-free inner solves. Pure block principal pivoting terminates only under a non-degeneracy
assumption; Júdice and Pires' least-index fallback already removes that requirement for exact solves~\cite{judice1994}, and Lemma~\ref{lem:inexact} shows the guarantee survives the inexact CG solves the matrix-free regime requires (Theorem~\ref{thm:termination}). The synthetic study of
Section~\ref{sec:results} then shows this unconditional guarantee to be dormant on generic data yet genuinely necessary meaning that on an adversarial family of anti-correlated designs the
unguarded batch path cycles, and only the fallback terminates it. Wrapping this guaranteed
loop around a matrix-free Krylov inner solver, warm-started across a parametric sweep, is the
approach that makes the guaranteed method a fast one as well. We develop and analyse it
here on its own, independent of any application.\footnote{A reference implementation is
released as the open-source package \texttt{nncg}:
\url{https://github.com/Jebel-Quant/nncg}. Its continuous-integration test suite is the
synthetic study of Section~\ref{sec:results}.}

Applying CG through a changing set of free variables has a lineage of its own, and the distinction from it is worth drawing precisely. Bertsekas' projected Newton
method~\cite{bertsekas1982}, Mor\'e and Toraldo's GPCG~\cite{moretoraldo1991}, and Dost\'al's
proportioning and MPRGP algorithms~\cite{dostal1997,dostalschoberl2005} are
\emph{feasible-point} methods. Every iterate satisfies the bounds, and the working face
changes through gradient projections. The switches between projection and CG phases are
governed by sufficient-decrease or proportioning tests. Their convergence rates are stated in
terms of $\kappa$, and they identify the optimal face finitely under a non-degeneracy
(strict-complementarity) assumption. They are moreover \emph{bound-constrained} solvers: a
general equality $Bx = c,$ a budget or factor-neutrality condition in the portfolio
application, breaks the closed-form orthant projection they rely on, and is reached only by
wrapping them in an outer augmented-Lagrangian loop~\cite{dostal2006smalbe}, at the price of an
extra level of iteration and of the exact certificate. The loop developed here is instead a
\emph{complementary-basis} method. Its iterates are infeasible between outer steps, the free
set changes by principal pivots driven by signs alone with no line search or decrease test, and
it takes $Bx = c$ directly through the Schur-complement elimination of
Section~\ref{sec:reduction}. This loop adds the equality $Bx = c$ but
treats only the non-negative orthant $x \geq 0$ (equivalently a single lower bound per variable), whereas the feasible-point family handles general two-sided box constraints
$\ell \leq x \leq u$ that our two-state complementary basis does not express. The termination is of combinatorial nature as it rests on Murty's least-index rule, which is what removes the non-degeneracy hypothesis. On the NNLS side, Bro and De Jong's fast NNLS~\cite{brodejong1997} accelerates Lawson--Hanson by caching normal-equations cross-products, but it still assembles $M^\top M$ and solves each working-set system by dense factorisation. The present method forms neither, reaching $A$ only through the two operator
applications of Section~\ref{sec:reduction}.

We briefly summarise the main structure of this paper. Section~\ref{sec:problem} states the problem and its KKT/complementarity structure.
Section~\ref{sec:reduction} reduces it to a sequence of unconstrained SPD solves, records the
matrix-free CG operator and its convergence rate, and eliminates the equality multipliers
analytically. Section~\ref{sec:nonneg} develops the active-set loop and the projected-gradient
reference method and proves finite termination. Section~\ref{sec:conditioning} treats
regularisation and preconditioning. Section~\ref{sec:results} records the complexity of each
method and discusses its behaviour on synthetic SPD test problems.

%% file: sections/s2_problem.tex
\section{Problem and Complementarity Structure}\label{sec:problem}

We consider the strictly convex non-negative quadratic program~\eqref{eq:bqp},
\[
  \min_{x \geq 0}\; f(x) = \tfrac{1}{2}\,x^\top A x - b^\top x,
  \qquad A = A^\top \succ 0,\quad b \in \mathbb{R}^n,
\]
together with its two structural specialisations. In the \emph{least-squares} form
$A = M^\top M$ and $b = M^\top d$ for $M \in \mathbb{R}^{m\times n}$ of full column rank, so
that $f(x) = \tfrac{1}{2}\|Mx - d\|_2^2 - \tfrac{1}{2}\|d\|_2^2$ and~\eqref{eq:bqp} is
NNLS~\cite{lawson1974}. In the \emph{equality-augmented} form a linear equality
system $Bx = c$ is imposed as well,
\begin{equation}\label{eq:eqp}
  \min_{x}\; \tfrac{1}{2}\,x^\top A x - b^\top x
  \quad\text{subject to}\quad Bx = c,\; x \geq 0,
  \qquad B \in \mathbb{R}^{p\times n},\; c \in \mathbb{R}^p,
\end{equation}
with $B$ of full row rank $p \leq n$. This constrains $x$ to the polytope
$\mathcal{P} = \{x \geq 0 : Bx = c\}$, an affine slice of the non-negative
orthant. The single normalisation $\mathbf{1}^\top x = \beta$ is the case $p=1$,
$B = \mathbf{1}^\top$, $c = \beta$. A general $B$ carries any finite family of
linear balance conditions (group budgets, factor neutralities, netting rules)
at the same cost structure.

\paragraph{Optimality and complementarity.}
Because $f$ is strictly convex and the feasible set of~\eqref{eq:bqp} is a non-empty closed
convex cone, a unique minimiser exists and the Karush-Kuhn-Tucker (KKT) conditions are
necessary and sufficient. We introduce a multiplier $s \geq 0$ for $x \geq 0$. Stationarity of
the Lagrangian $\tfrac{1}{2}x^\top A x - b^\top x - s^\top x$ gives the reduced gradient
$s = \nabla f(x) = A x - b$, and the KKT system is the linear complementarity problem
$\mathrm{LCP}(A,\,{-b})$:
\begin{equation}\label{eq:lcp}
  s = A x - b, \qquad x \geq 0, \qquad s \geq 0, \qquad x_i\,s_i = 0 \ \ \forall i.
\end{equation}
Complementary slackness partitions the indices into a \emph{free} set
$\mathcal{F} = \{i : x_i > 0\}$, on which $s_i = 0$ and hence $(Ax)_i = b_i$, and a
\emph{bound} set $\mathcal{B} = \{i : x_i = 0\}$, on which $s_i = (Ax)_i - b_i \geq 0$. A
bound variable with $s_i < 0$ would lower $f$ if released to a small positive value. We point out that a free
variable with $x_i < 0$ is infeasible and must be pushed to the bound. These two violations
drive the primal and dual steps of Section~\ref{sec:nonneg}.

For the equality-augmented problem~\eqref{eq:eqp} the Lagrangian carries an additional term
$-\lambda^\top(Bx - c)$ with a multiplier $\lambda \in \mathbb{R}^p$. The reduced gradient
becomes $s = A x - b - B^\top\lambda$, and the free-set stationarity reads
$(Ax)_i = b_i + (B^\top\lambda)_i$ for $i \in \mathcal{F}$. On the free set the gradient is
thus pinned to the row space of $B$, with $\lambda$ the vector of shadow prices of the $p$
balance conditions. The KKT system is now a \emph{mixed} complementarity problem, i.e. the free
block carries the equality $B_{\mathcal{F}}x_{\mathcal{F}} = c$ alongside the sign conditions.
Section~\ref{sec:reduction} shows that $\lambda$ need not be carried as an unknown in the
linear solve. It is recovered in closed form from $p+1$ SPD solves and a $p\times p$ Schur
system (for $p=1$, the two solves of the single-normalisation case). This requires
$B_{\mathcal{F}}$ to have full row rank. That holds whenever the free set is large enough,
$|\mathcal{F}| \geq p$, and generically thereafter; for $p=1$ it holds for every non-empty
free set.

\paragraph{The $P$-matrix property.}
Since $A \succ 0$, every principal submatrix denoted by $A_{\mathcal{F}} := A_{\mathcal{F},\mathcal{F}},$ where $\mathcal{F}$ is any index set, 
is itself SPD as $x_{\mathcal{F}}^\top A_{\mathcal{F}}\,x_{\mathcal{F}} = x^\top A x > 0$ for any
nonzero $x$ supported on $\mathcal{F}$. It is therefore invertible with
$\det A_{\mathcal{F}} > 0$, so the linear system
$A_{\mathcal{F}}\,x_{\mathcal{F}} = b_{\mathcal{F}}$ solved on any candidate free set is well
posed and CG applies at every outer step (Section~\ref{sec:reduction}). This much is
elementary and needs nothing but symmetry and definiteness of a principal submatrix.

The $P$-matrix label enters for one further purpose. A matrix all of whose principal minors
are positive is a $P$-matrix~\cite{murty1988}, and $A \succ 0$ is one; the point of the label
is that $\mathrm{LCP}(A,\,{-b})$ with a $P$-matrix has a unique solution reachable by principal
pivoting~\cite{murty1988,judice1994}, the backbone of the finite-termination guarantee in
Section~\ref{sec:nonneg}. When $A$ is only positive semidefinite (as in a rank-deficient
least-squares problem, $m < n$), the regularisation of Section~\ref{sec:conditioning} restores
$A \succ 0$ and with it the $P$-matrix property.

The complementarity structure~\eqref{eq:lcp} reduces the search for $x$ to a search over free
sets $\mathcal{F}$. Once $\mathcal{F}$ is fixed, the sign constraints drop out, and what
remains on $\mathcal{F}$ is exactly the linear system~\eqref{eq:innerspd} shown well posed
above. Section~\ref{sec:reduction} makes this reduction precise, together with the closed-form
elimination of $\lambda$ for the equality-augmented problem~\eqref{eq:eqp}.

%% file: sections/s3_reduction.tex
\section{Reduction to Unconstrained SPD Solves}\label{sec:reduction}

We now fix a candidate free set $\mathcal{F} \subseteq \{1,\dots,n\}$ and set $x_i = 0$ for
$i \notin \mathcal{F}$. On $\mathcal{F}$ the bound-constrained problem~\eqref{eq:bqp} reduces
to the \emph{unconstrained} strictly convex quadratic
\[
  \min_{x_{\mathcal{F}}}\; \tfrac{1}{2}\,x_{\mathcal{F}}^\top A_{\mathcal{F}}\,x_{\mathcal{F}}
    - b_{\mathcal{F}}^\top x_{\mathcal{F}},
\]
whose stationarity condition is the SPD linear system
\begin{equation}\label{eq:innerspd}
  A_{\mathcal{F}}\,x_{\mathcal{F}} = b_{\mathcal{F}}.
\end{equation}
The active-set loop of Section~\ref{sec:nonneg} solves a sequence of such systems, one per
outer step, and the inequality $x \geq 0$ is imposed by which indices $\mathcal{F}$ contains,
never inside the linear system solve. The inner solver is the proposed CG method.

\paragraph{Matrix-free operator.}
It is well known that CG needs only the mat-vec $v \mapsto A_{\mathcal{F}}\,v$, never the explicit matrix. For
$A = M^\top M$ this is $M_{\mathcal{F}}^\top(M_{\mathcal{F}}\,v)$ with
$M_{\mathcal{F}} = M_{:,\mathcal{F}}$, two products at $O(m\,|\mathcal{F}|)$ and no $O(n^2)$
storage, the standard CG-on-the-normal-equations
arrangement~\cite{golub2013,paige1975,choi2006} on a column subset that changes between outer
steps.

The whole method requires only two operations of $A$, and this mat-vec is the first. The
\emph{free-block application} $v \mapsto A_{\mathcal{F}}v$ is all CG needs to
solve~\eqref{eq:innerspd}. The \emph{cross product}
$x_{\mathcal{F}} \mapsto A_{\mathcal{B},\mathcal{F}}\,x_{\mathcal{F}}$ supplies the reduced
gradient $s_i = (Ax)_i - b_i$ at the bound indices $i \in \mathcal{B}$, which the
dual-feasibility test of Section~\ref{sec:nonneg} uses (since $x_{\mathcal{B}} = 0$). A full
product $x \mapsto Ax$ is convenient for a residual check but not otherwise needed. Any implementation realising these two operations runs the loop, its termination proof, and its
convergence analysis unchanged. A dense $A$ slices them directly. The Gram operator
$A = M^\top M$ evaluates them from $M$ without ever forming $A$, as above. A
diagonal-plus-low-rank operator $A = \operatorname{diag}(d) + U\Delta U^\top$ inverts the free
block by the Woodbury identity, a \emph{direct} inner solve that bypasses CG. The
equality-augmented problem~\eqref{eq:eqp} adds only the column selections
$v \mapsto B_{\mathcal{F}}v$ and $w \mapsto B_{\mathcal{F}}^\top w$ of the Schur elimination
below. When a variable is released in the primal step, the operator is rebuilt on the reduced
free set and CG restarts on the smaller system. We now state the standard convergence result
for CG.

\begin{proposition}[CG convergence~{\cite{greenbaum1997,trefethenbau1997}}]\label{prop:cg}
  Let $A_{\mathcal{F}}$ be SPD with spectral condition number
  $\kappa = \kappa(A_{\mathcal{F}})$, and let $x_k$ denote the $k$-th CG iterate
  for~\eqref{eq:innerspd}. Then
  \[
    \frac{\lVert e_k\rVert_{A_{\mathcal{F}}}}{\lVert e_0\rVert_{A_{\mathcal{F}}}}
    \;\leq\;
    2\!\left(\frac{\sqrt{\kappa}-1}{\sqrt{\kappa}+1}\right)^{\!k},
  \]
  where $\lVert e\rVert_{A_{\mathcal{F}}} = (e^\top A_{\mathcal{F}}\,e)^{1/2}$ is the
  $A_{\mathcal{F}}$-energy norm of the error $e = x_{\mathcal{F}} - x_k$.
\end{proposition}

Reaching a fixed relative energy-norm tolerance thus takes
$O(\sqrt\kappa)$ iterations, and clustered spectra, with a minimal condition number do better, the
superlinear phase the synthetic study of Section~\ref{sec:results} measures. Because
$A \succ 0$ forces every $A_{\mathcal{F}} \succ 0$, the bound applies at every outer step, and
the regularising split of Section~\ref{sec:conditioning} lowers $\kappa$. This $\sqrt\kappa$
dependence is the whole reason to keep a Krylov inner solver, while the projected-gradient reference
of Section~\ref{sec:nonneg} pays $O(\kappa)$ at the same per-step cost.

\paragraph{Eliminating the equality multipliers.}
For the equality-augmented problem~\eqref{eq:eqp}, stationarity and feasibility on the free
set are the $(|\mathcal{F}|+p)\times(|\mathcal{F}|+p)$ saddle-point system (cf.~\cite{benzi2005})
\begin{equation}\label{eq:saddle}
  \begin{pmatrix} A_{\mathcal{F}} & -B_{\mathcal{F}}^\top \\ B_{\mathcal{F}} & 0 \end{pmatrix}
  \begin{pmatrix} x_{\mathcal{F}} \\ \lambda \end{pmatrix}
  =
  \begin{pmatrix} b_{\mathcal{F}} \\ c \end{pmatrix},
  \qquad B_{\mathcal{F}} = B_{:,\mathcal{F}} \in \mathbb{R}^{p\times|\mathcal{F}|}.
\end{equation}
Rather than solving the indefinite system directly, we eliminate $\lambda$ using the Schur
complement. Stationarity gives $x_{\mathcal{F}} = A_{\mathcal{F}}^{-1}(b_{\mathcal{F}} +
B_{\mathcal{F}}^\top\lambda) = v_0 + V_1\lambda$ in terms of the SPD solves
\[
  v_0 = A_{\mathcal{F}}^{-1} b_{\mathcal{F}} \in \mathbb{R}^{|\mathcal{F}|},
  \qquad
  V_1 = A_{\mathcal{F}}^{-1} B_{\mathcal{F}}^\top \in \mathbb{R}^{|\mathcal{F}|\times p}
  \quad(\text{its $p$ columns, one CG solve each}).
\]
Imposing $B_{\mathcal{F}}x_{\mathcal{F}} = c$ then fixes the multiplier through the $p\times p$
system
\[
  S\,\lambda = c - B_{\mathcal{F}} v_0,
  \qquad
  S = B_{\mathcal{F}} V_1 = B_{\mathcal{F}} A_{\mathcal{F}}^{-1} B_{\mathcal{F}}^\top \succ 0,
  \qquad
  x_{\mathcal{F}} = v_0 + V_1\lambda,
\]
where $S$, the Schur complement, is SPD because $A_{\mathcal{F}} \succ 0$ and $B_{\mathcal{F}}$
has full row rank. The Lagrange multiplier is then $\lambda = S^{-1}(c - B_{\mathcal{F}}v_0)$,
and the $p\times p$ solve can be done by a dense Cholesky decomposition at negligible cost.

The single-normalisation case $p=1$, $B=\mathbf{1}^\top$, $c=\beta$ recovers the
two solves $v_0, v_1 = A_{\mathcal{F}}^{-1}\mathbf{1}$ with the scalar
$\lambda = (\beta - \mathbf{1}^\top v_0)/(\mathbf{1}^\top v_1)$. The pure-normalisation case
$b=0$, $\beta=1$ collapses further to the single solve
$x_{\mathcal{F}} = v_1/(\mathbf{1}^\top v_1)$.
Thus each outer step costs $p+1$ matrix-free CG solves. The multipliers are recovered
analytically and never appear in the linear solve, so every inner system stays SPD and
directly amenable to Proposition~\ref{prop:cg}. The $p+1$ right-hand sides share the operator
$A_{\mathcal{F}}$, and can be solved by a block-CG or by independent CG runs, warm-started
across outer steps.

%% file: sections/s5_nonneg.tex
\section{Enforcing Non-Negativity}\label{sec:nonneg}

Section~\ref{sec:reduction} reduces each candidate free set to an SPD solve, leaving
$x \geq 0$ as the sole remaining constraint. Two strategies enforce it. The first is an
active-set loop. It wraps the CG solver unchanged and certifies global optimality through the
complementarity system~\eqref{eq:lcp}. The second is projected gradient. It handles the
constraint inside every step but forgoes the Krylov space convergence rate of CG.

\subsection{Active-set / block-principal-pivoting loop}\label{ssec:activeset}

The active-set method maintains a free set $\mathcal{F}$, solves~\eqref{eq:innerspd} (or its
equality-augmented form) on $\mathcal{F}$ by CG, and adjusts $\mathcal{F}$ toward
feasibility. In the \emph{primal step} any free variable that returned negative is pushed to
the bound and dropped from $\mathcal{F}$. Once all free variables are non-negative, the
\emph{dual step} inspects the reduced gradient $s = Ax - b$ (minus $B^\top\lambda$ in the
equality-augmented case, with $\lambda$ from Section~\ref{sec:reduction}) at every bound
variable. A bound index $i$ with $s_i < 0$ would lower $f$ if released, so it is re-admitted
to $\mathcal{F}$. The loop terminates when no variable violates either test, that is, when
$x \geq 0$, $s \geq 0$, and complementary slackness hold to tolerance. These are exactly the
KKT system~\eqref{eq:lcp}, which for the strictly convex $f$ certifies the unique global
minimiser.

Toggling an index between $\mathcal{F}$ and $\mathcal{B}$ is a \emph{principal pivot} on
$\mathrm{LCP}(A,\,{-b})$, and dropping or adding several indices at once is a \emph{block}
principal pivot~\cite{judice1994,kim2011}. Because $A \succ 0$ is a $P$-matrix, every
principal submatrix is invertible and each pivot is well defined~\cite{murty1988}. Block
pivots are fast, but, like any all-at-once exchange, they can cycle: a batch drop can
over-shoot the support and a later batch add restore it, revisiting a free set already seen.
Algorithm~\ref{alg:elim} therefore runs the batch exchange as a \emph{fast path} guarded by an
anti-cycling counter. Let $N$ be the current number of primal and dual violators and $\bar N$
the fewest seen so far. A batch step is taken whenever it makes progress ($N < \bar N$) or a
patience budget $p$ is unspent. Otherwise the algorithm falls back to a single
\emph{Bland-style} pivot on the lowest-indexed violator. This least-index single pivot is
Murty's principal pivoting method, which cannot cycle~\cite{murty1988}. It is what the
finite-termination guarantee rests on. The batch path serves only to reach the optimum faster
when it does not stall. This is the anti-cycling construction of J\'udice and Pires for block
principal pivoting~\cite{judice1994}.

\begin{algorithm}[ht]
\caption{Active-set outer loop (wraps the CG inner solver $f$)}\label{alg:elim}
\begin{algorithmic}[1]
\Require Inner solver $f$ returning $(x_{\mathcal{F}}, k_{\mathrm{step}})$ for free set $\mathcal{F}$;\;
         SPD $A$ (or its mat-vec) and $b$;\; tolerance $\varepsilon$;\;
         patience $p_{\max} \in \mathbb{N}$ \textit{(default $3$)}
\Ensure Minimiser $x \in \mathbb{R}^n$ of~\eqref{eq:bqp};\; outer steps $s$;\; total inner iterations $k$
\State $\mathcal{F} \leftarrow \{1,\ldots,n\}$;\; $s \leftarrow 0$;\; $k \leftarrow 0$;\;
       $\bar N \leftarrow n+1$;\; $p \leftarrow p_{\max}$
\Loop
  \State $(x_{\mathcal{F}},\, k_{\mathrm{step}}) \leftarrow f(\mathcal{F})$;\;
         $s \leftarrow s + 1$;\; $k \leftarrow k + k_{\mathrm{step}}$
    \hfill\textit{($k_{\mathrm{step}}$: CG iterations; $1$ for a direct inner solve)}
  \State $x_i \leftarrow (x_{\mathcal{F}})_i$ for $i \in \mathcal{F}$;\; $x_i \leftarrow 0$ for $i \notin \mathcal{F}$
  \State $s_i \leftarrow (Ax)_i - b_i$
    \hfill\textit{(reduced gradient; subtract $(B^\top\lambda)_i$ in the equality-augmented case)}
  \State $\mathcal{D} \leftarrow \{i \in \mathcal{F} : x_i < -\varepsilon\}$;\;
         $\mathcal{V} \leftarrow \{i \notin \mathcal{F} : s_i < -\varepsilon\}$
    \hfill\textit{(primal and dual violators)}
  \State $\mathcal{H} \leftarrow \mathcal{D} \cup \mathcal{V}$;\; $N \leftarrow |\mathcal{H}|$
  \If{$N = 0$}
    \State \textbf{break} \hfill\textit{(KKT~\eqref{eq:lcp} satisfied; globally optimal)}
  \ElsIf{$N < \bar N$ \textbf{ or } $p > 0$}
    \hfill\textit{(fast path: progress, or patience remains)}
    \State \textbf{if} $N < \bar N$ \textbf{ then } $\bar N \leftarrow N$;\; $p \leftarrow p_{\max}$ \textbf{ else } $p \leftarrow p - 1$
    \State $\mathcal{F} \leftarrow (\mathcal{F} \setminus \mathcal{D}) \cup \mathcal{V}$
      \hfill\textit{(batch exchange: drop all $\mathcal{D}$, add all $\mathcal{V}$)}
  \Else
    \hfill\textit{(anti-cycling fallback: $p$ exhausted)}
    \State $i^\star \leftarrow \min\{i : i \in \mathcal{H}\}$
      \hfill\textit{(Bland least-index rule)}
    \State \textbf{if} $i^\star \in \mathcal{D}$ \textbf{ then } $\mathcal{F} \leftarrow \mathcal{F} \setminus \{i^\star\}$
           \textbf{ else } $\mathcal{F} \leftarrow \mathcal{F} \cup \{i^\star\}$
      \hfill\textit{(single principal pivot)}
  \EndIf
\EndLoop
\State \Return $x$, $s$, $k$
\end{algorithmic}
\end{algorithm}

\begin{theorem}[Finite convergence of Algorithm~\ref{alg:elim}]\label{thm:termination}
Let $f(x) = \tfrac{1}{2}x^\top A x - b^\top x$ with $A \succ 0$. Then for any patience budget
$p_{\max} \in \mathbb{N}$, Algorithm~\ref{alg:elim} terminates in finitely many outer
iterations and returns the unique global minimiser of $\min_{x \geq 0} f(x)$, even under degeneracy. The same holds for the equality-augmented
problem~\eqref{eq:eqp} with the reduced gradient of line~5 adjusted by $B^\top\lambda$,
provided $B_{\mathcal{F}}$ retains full row rank on every free set the loop visits (automatic
for $p=1$; generic once $|\mathcal{F}| \geq p$).
\end{theorem}

\begin{proof}
\emph{Step 1 (each inner solve is well posed).}
Because $A \succ 0$, every principal submatrix $A_{\mathcal{F}}$ is SPD (hence $A$ is a
$P$-matrix~\cite{murty1988}, used in Step~3), so the system
$A_{\mathcal{F}}x_{\mathcal{F}} = b_{\mathcal{F}}$ has a unique solution to which CG converges
(Proposition~\ref{prop:cg}). In the
equality-augmented case the free-set solve is the saddle system~\eqref{eq:saddle}; with
$A_{\mathcal{F}} \succ 0$ and $B_{\mathcal{F}}$ of full row rank its Schur complement $S$ is
SPD, so~\eqref{eq:saddle} has the unique solution recovered by the $p+1$ SPD solves of
Section~\ref{sec:reduction}.

\emph{Step 2 (termination implies global optimality).}
Free-set stationarity gives $s_i = 0$ for $i \in \mathcal{F}$. The loop exits only when
$N = 0$: (i) all free $x_i \geq 0$ and (ii) all bound $s_i \geq 0$. These are exactly
$x \geq 0$, $s \geq 0$, $x_i s_i = 0$: the KKT system~\eqref{eq:lcp}. Strict convexity of
$f$ makes them sufficient for the unique global minimiser; it therefore suffices to show the
loop cannot run forever.

\emph{Step 3 (the fallback cannot cycle).}
Call a maximal block of consecutive iterates taking the \textbf{else} branch (the single Bland
pivot) a \emph{fallback run}. Within a run $\bar N$ is constant, and each iterate performs one
principal pivot on the lowest-indexed violator. This is precisely Murty's least-index
principal pivoting method for the $P$-matrix LCP of Step~2, which generates a sequence of
complementary bases (here, free sets) in which no basis recurs and which reaches the solution
in finitely many pivots~\cite{murty1988}. A fallback run is an initial segment of that
sequence started from the free set left by the preceding step, so each run is finite,
wherever it starts.

\emph{Step 4 (only finitely many batch steps and runs).}
The counter satisfies $0 \le N \le n$, and $\bar N$ is non-increasing, strictly decreasing
exactly when the progress branch ($N < \bar N$) fires; hence that branch fires at most $n+1$
times. Each remaining fast-path step has $N \ge \bar N$ and decrements $p$, and $p$ resets only
on a progress step, so at most $p_{\max}$ such steps occur between progress events. The number
of fallback runs is likewise bounded: a run ends at termination ($N=0$) or at a progress step,
and there are at most $n+1$ progress steps. Thus the algorithm takes at most
$(n+1) + (n+2)\,p_{\max}$ fast-path steps and at most $n+2$ fallback runs.

\emph{Step 5 (finite termination).}
By Step~3 each of the at most $n+2$ fallback runs is finite, and by Step~4 the fast-path steps
are finite in number; their sum is finite, so the loop reaches $N=0$ after finitely many outer
iterations, and Step~2 certifies the returned $x$ as the unique global minimiser. A crude
ceiling is the number of complementary bases, $2^n$.
\end{proof}

The $2^n$ ceiling establishes \emph{finiteness} only; it is not an efficiency estimate. The
guarantee is unconditional. The least-index fallback removes the non-degeneracy hypothesis
that earlier active-set arguments need, and it is the fallback, not the batch exchange, that
the proof rests on. In practice the fallback is rarely reached. A batch step fails to
reduce $N$ only when a dropped index is immediately re-added (or vice versa), which forces
$s_i = 0$ exactly for some toggled $i$, a codimension-one and hence non-generic
condition~\cite{judice1994}. Block principal pivoting therefore converges in a handful of
outer steps on typical data, with the fast path alone certifying optimality.
Section~\ref{sssec:fallback} exhibits the flip side: an anti-correlated design family that
makes the event systematic, on which the unguarded batch path provably cycles and only the
fallback terminates.

\paragraph{What is new relative to Júdice--Pires.}
For exact inner solves, the guarded loop above \emph{is} the block-principal-pivoting
algorithm of Júdice and Pires~\cite{judice1994}, namely a batch exchange safeguarded by Murty's
least-index single pivot~\cite{murty1988}, already finite on the strictly monotone
($A \succ 0$) LCP, even under degeneracy. We claim no novelty for that
construction and point out that Theorem~\ref{thm:termination} restates it in active-set language for
self-containedness. The contribution is what Table~\ref{tab:novelty} isolates, and its crux is
the first row where classical block-pivoting termination assumes each principal-block system is
solved \emph{exactly}, pointing out that the  matrix-free regime solves it to a specified tolerance.
Lemma~\ref{lem:inexact} closes this gap for the bound-constrained problem~\eqref{eq:bqp}, and
Corollary~\ref{cor:inexact-eq} extends it through the multiplier recovery of
Section~\ref{sec:reduction} to the equality-augmented problem~\eqref{eq:eqp}: once the inner
residual is tied to the decision margin, the inexact loop makes the identical primal and dual
sign decisions as the exact loop, so the Júdice--Pires count transfers verbatim. Without it,
block pivoting and Krylov inner solves do not obviously compose: a tolerance-level solve can
flip a sign test and reintroduce the cycling the fallback exists to prevent.

\begin{table}[ht]\centering\footnotesize
\begin{tabular}{lll}
\toprule
 & Júdice--Pires~\cite{judice1994} / Murty~\cite{murty1988} & This paper \\
\midrule
Inner solve  & exact principal-block solve & inexact matrix-free CG, guarantee preserved (Lemma~\ref{lem:inexact}, Cor.~\ref{cor:inexact-eq}) \\
Operator     & assembled/dense             & never forms $A$; $O(\sqrt\kappa)$ Krylov rate, $O(n)$ memory \\
Constraints  & bounds ($x\ge0$)            & bounds \emph{and} equality $Bx=c$ (Schur complement) \\
SPD premise  & assumed                     & secured for rank-deficient $A=M^\top M$ by the regularising split \\
\bottomrule
\end{tabular}
\caption{What Theorem~\ref{thm:termination} and Lemma~\ref{lem:inexact} add over the
finite-termination guarantee of block principal pivoting. The termination \emph{count} is
Júdice--Pires'; the novelty is carrying it to the inexact, matrix-free, equality-augmented
setting.}
\label{tab:novelty}
\end{table}

Theorem~\ref{thm:termination} assumes each inner call returns the exact minimiser of
$f_{\mathcal{F}}$, whereas CG returns an iterate accurate only to its residual tolerance. The
following lemma closes that gap. Once the residual is tight enough relative to the test
threshold $\varepsilon$, every primal and dual sign decision of Algorithm~\ref{alg:elim}
coincides with the decision the exact solve would make. The inexact loop then inherits the results of the
theorem.

\begin{lemma}[Inexact inner solves]\label{lem:inexact}
Fix the bound-constrained problem~\eqref{eq:bqp}. For a free set $\mathcal{F}$ let
$\hat x(\mathcal{F})$ denote the exact solution of~\eqref{eq:innerspd} extended by zeros to
$\mathbb{R}^n$, and $\hat s(\mathcal{F}) = A\,\hat x(\mathcal{F}) - b$ its reduced gradient.
Let $\mathcal{F}_0 = \{1,\dots,n\}, \mathcal{F}_1, \dots, \mathcal{F}_T$ be the free-set
trajectory of Algorithm~\ref{alg:elim} run with \emph{exact} inner solves, finite by
Theorem~\ref{thm:termination}. Define the \emph{decision margin}
\[
  \mu \;=\; \min_{0 \le t \le T} \min\Big(
    \big\{\, |\hat x_i(\mathcal{F}_t)| : i \in \mathcal{F}_t,\ \hat x_i(\mathcal{F}_t) \neq 0 \,\big\}
    \,\cup\,
    \big\{\, |\hat s_i(\mathcal{F}_t)| : i \notin \mathcal{F}_t,\ \hat s_i(\mathcal{F}_t) \neq 0 \,\big\}
  \Big),
\]
the smallest nonzero magnitude any primal or dual test encounters along the exact
trajectory; $\mu > 0$ as a minimum of finitely many positive numbers. Suppose the test
tolerance satisfies $0 < \varepsilon < \mu$ and every inner CG call is stopped at a residual
$\rho = \lVert b_{\mathcal{F}} - A_{\mathcal{F}}\tilde x_{\mathcal{F}} \rVert_2$ obeying
\begin{equation}\label{eq:inexactcond}
  \rho
  \left(\frac{1 + \lambda_{\max}(A)}{\lambda_{\min}(A)}\right)
  \;<\; \min\big(\varepsilon,\ \mu - \varepsilon\big).
\end{equation}
Then at every outer step the violator sets $\mathcal{D}, \mathcal{V}$ of
Algorithm~\ref{alg:elim} computed from the CG iterate coincide with those computed from the
exact solve on the same free set. The inexact loop therefore traverses exactly the free-set
trajectory $\mathcal{F}_0,\dots,\mathcal{F}_T$ of the exact loop, Theorem~\ref{thm:termination}
applies verbatim, and the returned iterate satisfies
$\lVert \tilde x - x^\star \rVert_2 \le \rho/\lambda_{\min}(A)$, hence also
$\lVert \tilde x - x^\star \rVert_\infty \le \rho/\lambda_{\min}(A)$.
\end{lemma}

\begin{proof}
We write $r = b_{\mathcal{F}} - A_{\mathcal{F}}\tilde x_{\mathcal{F}}$, so
$\tilde x_{\mathcal{F}} - \hat x_{\mathcal{F}} = -A_{\mathcal{F}}^{-1} r$. From the Cauchy
interlacing theorem, $\lambda_{\min}(A_{\mathcal{F}}) \ge \lambda_{\min}(A)$, hence, in the
$2$-norm throughout,
\[
  \delta_x \;:=\; \lVert \tilde x_{\mathcal{F}} - \hat x_{\mathcal{F}} \rVert_2
  \;\le\; \lVert A_{\mathcal{F}}^{-1} r \rVert_2
  \;\le\; \rho/\lambda_{\min}(A).
\]
Since $|\tilde x_i - \hat x_i| \le \lVert \tilde x_{\mathcal{F}} - \hat x_{\mathcal{F}}
\rVert_2$ for every coordinate $i$, $\delta_x$ also bounds each individual primal test below.
For a bound index $i \notin \mathcal{F}$ the reduced gradients differ by
$\tilde s_i - \hat s_i = A_{i,\mathcal{F}}(\tilde x_{\mathcal{F}} - \hat x_{\mathcal{F}})$,
where $A_{i,\mathcal{F}}$ is the $i$th row of $A$ restricted to the free set. By
Cauchy--Schwarz and $\lVert A_{i,\mathcal{F}} \rVert_2 \le \lVert A \rVert_2 =
\lambda_{\max}(A)$, $\delta_s := \max_{i \notin \mathcal{F}} |\tilde s_i - \hat s_i| \le
\lambda_{\max}(A)\,\delta_x \le \lambda_{\max}(A)\,\rho/\lambda_{\min}(A)$. Both
perturbations are bounded by
$\delta := \rho\,(1+\lambda_{\max}(A))/\lambda_{\min}(A) < \min(\varepsilon, \mu-\varepsilon)$
by~\eqref{eq:inexactcond}, so $\delta_x, \delta_s \le \delta$.

We now fix $i \in \mathcal{F}_t$ at some step $t$ (writing $\mathcal{F} = \mathcal{F}_t$). If
$\hat x_i \ge 0$ then $\hat x_i = 0$ or $\hat x_i \ge \mu$ by definition of the margin (the
zero case is not excluded); either way $\tilde x_i \ge
-\delta > -\varepsilon$, so the inexact test does not flag $i$, matching the exact test. If
$\hat x_i < 0$ then $\hat x_i \le -\mu$, so $\tilde x_i \le -\mu + \delta < -\varepsilon$: both
tests flag $i$. The same argument, with $\hat s_i, \tilde s_i, \delta_s$ in place of
$\hat x_i, \tilde x_i, \delta_x$, shows $\tilde s_i \ge -\varepsilon$ whenever $\hat s_i \ge 0$
and $\tilde s_i < -\varepsilon$ whenever $\hat s_i < 0$, for every bound index $i \notin
\mathcal{F}$ --- the only comparison the dual test performs, since line~6 of
Algorithm~\ref{alg:elim} tests $s_i$ only off the free set, where the exact reduced gradient
need not vanish. (On $\mathcal{F}$ itself $\hat s_i = 0$ by stationarity and $\tilde s_i$ is
never tested, so no parallel case arises there.) Hence $\tilde{\mathcal{D}} =
\hat{\mathcal{D}}$ and $\tilde{\mathcal{V}} = \hat{\mathcal{V}}$ at step $t$; since both loops
start from $\mathcal{F}_0 = \{1,\dots,n\}$ and the counters $\bar N, p$ are deterministic
functions of the violator sequence, induction on $t$ gives the same free-set trajectory
$\mathcal{F}_0,\dots,\mathcal{F}_T$ and the same stopping step. At termination the exact
iterate is $x^\star$ (Theorem~\ref{thm:termination}), so
$\lVert \tilde x - x^\star \rVert_2 = \delta_x \le \rho/\lambda_{\min}(A)$, the off-set
entries being zero in both; $\lVert\cdot\rVert_\infty \le \lVert\cdot\rVert_2$ gives the same
bound in the $\infty$-norm.
\end{proof}

Lemma~\ref{lem:inexact} covers problem~\eqref{eq:bqp}, where the reduced gradient is
$s = Ax - b$ directly. In the equality-augmented problem~\eqref{eq:eqp}, $s$ carries the
multiplier term $-B^\top\lambda$, and $\lambda$ is itself recovered from the $p+1$ inexact CG
solves of Section~\ref{sec:reduction} rather than read off a single residual; the next result
propagates that second channel of inexactness through the Schur-complement recovery of
$\lambda$ before repeating the sign-test argument above.

\begin{corollary}[Inexact inner solves, equality-augmented case]\label{cor:inexact-eq}
Fix a free set $\mathcal{F}$ with $B_{\mathcal{F}}$ of full row rank, and let $v_0, V_1, S,
\lambda$ be as in Section~\ref{sec:reduction}, so $\hat x_{\mathcal{F}} = v_0 + V_1\lambda$.
Suppose the $p+1$ inner CG solves return $\tilde v_0, \tilde V_1$ at residuals
\[
  \rho_0 = \lVert b_{\mathcal{F}} - A_{\mathcal{F}}\tilde v_0 \rVert_2, \qquad
  \rho_1 = \lVert B_{\mathcal{F}}^\top - A_{\mathcal{F}}\tilde V_1 \rVert_2
\]
(the second the spectral norm of the residual matrix stacked over the $p$ column solves), that
the multiplier is recovered as $\tilde\lambda = \tilde S^{-1}(c - B_{\mathcal{F}}\tilde v_0)$
with $\tilde S = B_{\mathcal{F}}\tilde V_1$, and that
\begin{equation}\label{eq:eqresid}
  \rho_1\,\lVert B_{\mathcal{F}} \rVert_2 \;\le\; \tfrac{1}{2}\,\lambda_{\min}(A)\,\lambda_{\min}(S).
\end{equation}
Then
\[
  \delta_\lambda \;:=\; \lVert \tilde\lambda - \lambda \rVert_2 \;\le\;
  \frac{2\lVert B_{\mathcal{F}} \rVert_2}{\lambda_{\min}(S)\,\lambda_{\min}(A)}
  \big(\rho_0 + \rho_1 \lVert \lambda \rVert_2\big),
\]
and the recovered iterate $\tilde x_{\mathcal{F}} = \tilde v_0 + \tilde V_1\tilde\lambda$
satisfies
\[
  \delta_x^{\mathrm{eq}} \;:=\; \lVert \tilde x_{\mathcal{F}} - \hat x_{\mathcal{F}} \rVert_2
  \;\le\; \frac{\rho_0 + \rho_1\lVert\lambda\rVert_2}{\lambda_{\min}(A)}
  \;+\; \frac{\lVert B_{\mathcal{F}} \rVert_2 + \rho_1}{\lambda_{\min}(A)}\,\delta_\lambda.
\]
For a bound index $i \notin \mathcal{F}$, the reduced gradient $s_i = (Ax)_i - b_i -
(B^\top\lambda)_i$ perturbs by
\[
  \delta_s^{\mathrm{eq}} \;:=\; \max_{i \notin \mathcal{F}} |\tilde s_i - \hat s_i| \;\le\;
  \lambda_{\max}(A)\,\delta_x^{\mathrm{eq}} + \lVert B \rVert_2\,\delta_\lambda.
\]
If $\mu$ is redefined as in Lemma~\ref{lem:inexact} but along the exact free-set trajectory of
the \emph{equality-augmented} loop (Theorem~\ref{thm:termination}'s second clause), and
$\max(\delta_x^{\mathrm{eq}}, \delta_s^{\mathrm{eq}}) < \min(\varepsilon, \mu - \varepsilon)$ at
every outer step, then the sign-test argument in the proof of Lemma~\ref{lem:inexact} applies
verbatim with $\delta_x, \delta_s$ replaced by $\delta_x^{\mathrm{eq}}, \delta_s^{\mathrm{eq}}$:
the inexact equality-augmented loop traverses the same free-set trajectory as the exact one,
and Theorem~\ref{thm:termination}'s equality-augmented guarantee transfers verbatim.
\end{corollary}

\begin{proof}
Write $R_1 = B_{\mathcal{F}}^\top - A_{\mathcal{F}}\tilde V_1$, so $\tilde V_1 - V_1 =
-A_{\mathcal{F}}^{-1}R_1$ and, exactly as in Lemma~\ref{lem:inexact}'s proof,
$\lVert \tilde v_0 - v_0 \rVert_2 \le \rho_0/\lambda_{\min}(A)$ and
$\lVert \tilde V_1 - V_1 \rVert_2 \le \rho_1/\lambda_{\min}(A)$. Let
$\Delta S = \tilde S - S = B_{\mathcal{F}}(\tilde V_1 - V_1)$ and $\Delta d = \tilde d - d$
where $d = c - B_{\mathcal{F}}v_0$; then
$\lVert \Delta S \rVert_2 \le \lVert B_{\mathcal{F}} \rVert_2\,\rho_1/\lambda_{\min}(A)$ and
$\lVert \Delta d \rVert_2 \le \lVert B_{\mathcal{F}} \rVert_2\,\rho_0/\lambda_{\min}(A)$.
Condition~\eqref{eq:eqresid} gives $\lVert S^{-1}\rVert_2 \lVert\Delta S\rVert_2 \le 1/2$, so
the Neumann series for $\tilde S^{-1} = S^{-1}(I + \Delta S\,S^{-1})^{-1}$ gives
$\lVert \tilde S^{-1} \rVert_2 \le 2\lVert S^{-1} \rVert_2 = 2/\lambda_{\min}(S)$. Since
$\tilde S^{-1} - S^{-1} = -\tilde S^{-1}\Delta S\,S^{-1}$ and $S^{-1}d = \lambda$,
\[
  \tilde\lambda - \lambda = \tilde S^{-1}\tilde d - S^{-1}d
  = \tilde S^{-1}\Delta d + (\tilde S^{-1}-S^{-1})d
  = \tilde S^{-1}\big(\Delta d - \Delta S\,\lambda\big),
\]
and taking norms gives $\delta_\lambda$. For $\delta_x^{\mathrm{eq}}$,
\[
  \tilde x_{\mathcal{F}} - \hat x_{\mathcal{F}} = (\tilde v_0 - v_0) +
  \tilde V_1(\tilde\lambda - \lambda) + (\tilde V_1 - V_1)\lambda,
\]
and, using $\lVert V_1 \rVert_2 = \lVert A_{\mathcal{F}}^{-1}B_{\mathcal{F}}^\top \rVert_2 \le
\lVert B_{\mathcal{F}} \rVert_2/\lambda_{\min}(A)$,
$\lVert \tilde V_1 \rVert_2 \le \lVert V_1 \rVert_2 + \rho_1/\lambda_{\min}(A) \le
(\lVert B_{\mathcal{F}} \rVert_2 + \rho_1)/\lambda_{\min}(A)$; the triangle inequality gives
the stated bound. Finally $\tilde s_i - \hat s_i = A_{i,\mathcal{F}}(\tilde x_{\mathcal{F}} -
\hat x_{\mathcal{F}}) - B_{:,i}^\top(\tilde\lambda - \lambda)$, where $B_{:,i}$ is the $i$th
column of $B$ (so $B_{:,i}^\top\lambda = (B^\top\lambda)_i$), bounded as in
Lemma~\ref{lem:inexact}'s proof by $\lambda_{\max}(A)\,\delta_x^{\mathrm{eq}} + \lVert B
\rVert_2\,\delta_\lambda$. The sign-test agreement and trajectory transfer then follow exactly
as in the proof of Lemma~\ref{lem:inexact}, with $\delta_x^{\mathrm{eq}},
\delta_s^{\mathrm{eq}}$ in place of $\delta_x, \delta_s$.
\end{proof}

\subsection{Projected-gradient reference method}\label{ssec:proj}

The second strategy enforces the constraint at every step. Projected gradient iterates
\[
  x_{k+1} = \Pi_C\!\left(x_k - \tau\,\nabla f(x_k)\right),
  \qquad \nabla f(x) = A x - b,
\]
where $C$ is the feasible set. For~\eqref{eq:bqp} it is the non-negative orthant, so
$\Pi_C(y) = \max(y, 0)$ componentwise; for the single normalisation $\mathbf{1}^\top x=\beta$,
$x\geq0$ it is the simplex slice $\Delta_\beta$, projected onto in $O(n\log n)$ by a single
sort-and-threshold pass~\cite{duchi2008,condat2016}. For a general equality system $Bx=c$,
however, the projection onto $\mathcal{P} = \{x\geq0: Bx=c\}$ is itself a non-negative
quadratic program with no closed form, so projected gradient loses its per-step simplicity.
This is one reason the active-set loop is the method of choice once $p>1$ as its cost rises only
by the $p$ extra right-hand sides of Section~\ref{sec:reduction}. The step size is $\tau = 1/L$ with
$L = \lambda_{\max}(A)$, estimated once by power iteration on the same (matrix-free) operator.
There is no outer loop and no linear solve: each step is one gradient evaluation (two products
with $M$ when $A = M^\top M$) plus a projection.

\begin{proposition}[Projected-gradient convergence~{\cite{nesterov2004}}]\label{prop:proj}
  Let $A \succ 0$, so $f$ is $\mu$-strongly convex and $L$-smooth with
  $\mu = \lambda_{\min}(A)$ and $L = \lambda_{\max}(A)$, and let
  $x^\star = \arg\min_{x \in C} f(x)$. Projected gradient with $\tau = 1/L$ satisfies
  \[
    \|x_k - x^\star\|^2 \;\leq\; \left(1 - \tfrac{1}{\kappa(A)}\right)^{\!k}\|x_0 - x^\star\|^2,
    \qquad k \geq 0.
  \]
\end{proposition}
\begin{proof}[Proof sketch]
$x^\star = \Pi_C(x^\star - \tau\nabla f(x^\star))$ is a fixed point; non-expansiveness of
$\Pi_C$ and co-coercivity of the $L$-smooth gradient give
$\|x_{k+1}-x^\star\|^2 \le \|x_k-x^\star\|^2 - \tau(2-\tau L)\langle \nabla f(x_k)-\nabla f(x^\star),\,x_k-x^\star\rangle$.
Setting $\tau = 1/L$ makes $2-\tau L = 1$, and $\mu$-strong convexity bounds the inner product
below by $\mu\|x_k-x^\star\|^2$, so $\|x_{k+1}-x^\star\|^2 \le (1-\mu/L)\|x_k-x^\star\|^2$.
\end{proof}

The iteration count scales as $O(\kappa)$, against $O(\sqrt\kappa)$ for the CG inner loop
(Proposition~\ref{prop:cg}). Both methods pay the same per-step cost (one operator
application), so the $\sqrt\kappa$ gap in iteration counts is a runtime gap on
ill-conditioned problems. Nesterov acceleration would lift the projected-gradient rate to
$O(\sqrt\kappa)$, but CG attains that rate \emph{optimally} within the Krylov subspace. Each
CG iterate minimises $f$ over $\mathcal{K}_k(A_{\mathcal{F}}, b_{\mathcal{F}})$, a strictly
richer search space than two-point momentum, and typically with a smaller constant. Plain
projected
gradient is therefore the honest first-order reference method, isolating the Krylov advantage; a
FISTA variant~\cite{beck2009} narrows the constant but not the asymptotic gap. The trade-off
is otherwise in the other direction: projected gradient is loop-free and needs no
complementarity certificate, so it is the simpler method when $\kappa$ is small or moderate
accuracy suffices.

\subsection{Warm-starting}\label{ssec:warm}

For a parametric family of problems (e.g.\ a sequence $b(\theta_1), b(\theta_2), \dots$ of
right-hand sides, or a swept normalisation $\beta$), consecutive minimisers usually differ in
only a few active constraints. Passing the previous problem's $(\mathcal{F}, x)$ pair as the
initial state of the next solve reduces both the outer count (the free set needs fewer primal
and dual adjustments) and the inner count (CG starts from a small initial residual, its guess
being the previous solution restricted to the new free set). The support-stable case admits a
clean statement.

\begin{proposition}[Warm starts across a support-stable step]\label{prop:warm}
Let $b' = b + \delta$ and suppose the free set $\mathcal{F}$ returned by
Algorithm~\ref{alg:elim} for $b$ is also optimal for $b'$, i.e.\ the exact solve
$\hat x_{\mathcal{F}} = A_{\mathcal{F}}^{-1} b'_{\mathcal{F}}$ is non-negative and its reduced
gradient satisfies $s \geq 0$ on the bound set. Then Algorithm~\ref{alg:elim} started from
$\mathcal{F}$ terminates in a \emph{single} outer step, and its one CG solve, warm-started at
the previous solution $x_{\mathcal{F}}$, reaches energy-norm accuracy $\tau$ in at most
\[
  k \;\leq\; \left\lceil \frac{\sqrt{\kappa}+1}{2}\,
  \ln\!\frac{2\,\lVert \delta \rVert_2}{\tau\sqrt{\lambda_{\min}(A)}} \right\rceil
  \quad\text{iterations},
\]
the inner count scales with the logarithm of the \emph{parameter step} $\lVert\delta\rVert$,
not of the solution norm.
\end{proposition}

\begin{proof}
The single solve on $\mathcal{F}$ returns $\hat x_{\mathcal{F}} \geq 0$ with $s \geq 0$ on the
bound set, so both violator sets of line~5 are empty, the loop exits, and Step~2 of
Theorem~\ref{thm:termination}'s proof certifies the global minimiser. For the iteration
bound, the warm start's initial error is
$e_0 = \hat x_{\mathcal{F}} - x_{\mathcal{F}} = A_{\mathcal{F}}^{-1}\delta_{\mathcal{F}}$,
whose energy norm is
$\lVert e_0 \rVert_{A_{\mathcal{F}}} = (\delta_{\mathcal{F}}^\top A_{\mathcal{F}}^{-1}
\delta_{\mathcal{F}})^{1/2} \leq \lVert\delta\rVert_2 / \sqrt{\lambda_{\min}(A)}$ by Cauchy
interlacing. Proposition~\ref{prop:cg} gives
$\lVert e_k \rVert_{A_{\mathcal{F}}} \leq 2\rho^k \lVert e_0 \rVert_{A_{\mathcal{F}}}$ with
$\rho = (\sqrt\kappa - 1)/(\sqrt\kappa + 1)$, and
$\ln(1/\rho) \geq 2/(\sqrt\kappa + 1)$; solving $2\rho^k\lVert e_0\rVert \leq \tau$ for $k$
yields the bound.
\end{proof}

When the support drifts, no comparable a-priori bound is available. The loop is not
monotone in the free set, as the adversarial family of Section~\ref{sssec:fallback} makes
plain. The empirical picture is nonetheless simple: the warm-started outer count tracks the
support drift $|\mathcal{F} \,\triangle\, \mathcal{F}'|$ and collapses to one whenever the
drift is zero, exactly as the proposition prescribes (Section~\ref{ssec:frontier}). A direct
inner solver profits only from the warm free set, having no analogue of an initial guess. The
matrix-free CG inner solver profits from both, which is where warm-started parametric sweeps
obtain their largest speed-ups.

%% file: sections/s6_conditioning.tex
\section{Regularisation and Preconditioning}\label{sec:conditioning}

Both the inner CG rate (Proposition~\ref{prop:cg}) and the projected-gradient rate
(Proposition~\ref{prop:proj}) are governed by $\kappa = \kappa(A)$, so anything that
compresses the spectrum of $A$ shortens the solve. Two mechanisms are available: a
\emph{regularising split}, which moves the operator toward a target the application deems
preferable, and \emph{preconditioning}, which accelerates the solve of the operator as given.
They coincide in a convenient way, developed in turn below.

\subsection{Regularisation}\label{ssec:reg}

We replace $A$ by the convex combination
\begin{equation}\label{eq:regsplit}
  A_\alpha = (1-\alpha)\,A + \alpha\,R^\top R, \qquad \alpha \in (0,1),\; R^\top R \succ 0,
\end{equation}
a ridge/Tikhonov regularisation of $A$ toward the SPD target $R^\top R$. When $A = M^\top M$
this is again a Gram matrix, of the stacked operator
\[
  \tilde{M} = \begin{pmatrix} \sqrt{1-\alpha}\,M \\ \sqrt{\alpha}\,R \end{pmatrix},
  \qquad \tilde{M}^\top\tilde{M} = A_\alpha,
\]
so it is simply the Gram backend of $\tilde M$ (Section~\ref{sec:reduction})
and CG runs on $\tilde{M}_{\mathcal{F}}$ with no change. The
augmented-matrix device is standard in least-squares
regularisation~\cite{golub2013,lawson1974}.

\begin{proposition}[Condition number under the regularising split]\label{prop:regcond}
  For $A_\alpha$ as in~\eqref{eq:regsplit}, Weyl's inequality~\cite{hornjohnson2013} gives
  \[
    \kappa(A_\alpha) \;\leq\;
    \frac{(1-\alpha)\lambda_{\max}(A) + \alpha\,\lambda_{\max}(R^\top R)}
         {\alpha\,\lambda_{\min}(R^\top R)}.
  \]
  It decreases monotonically in $\alpha$; for the scaled-identity target
  $R^\top R = \bar\lambda I$ it is
  $\kappa(A_\alpha) \le \tfrac{1-\alpha}{\alpha}\,\tfrac{\lambda_{\max}(A)}{\bar\lambda} + 1
  = O\!\big((1-\alpha)/\alpha\big)$.
\end{proposition}

The split works in two ways. First, it lowers $\kappa$ and hence the iteration count of both
solvers (Propositions~\ref{prop:cg} and~\ref{prop:proj}). Second, because $A_\alpha \succ 0$
strictly whenever $\alpha > 0$, it restores the $P$-matrix property on which
Theorem~\ref{thm:termination} depends, even when $A$ is only positive semidefinite (a
rank-deficient least-squares operator with $m < n$). A strictly positive $\alpha$ is therefore the natural
setting for the active-set loop: it simultaneously conditions the inner solve and guarantees
that every principal submatrix is invertible.

\subsection{Preconditioning}\label{ssec:precond}

Regularisation lowers $\kappa$ by changing the operator. When no modelling change is
admissible and $A$ must be solved as given, the same reduction in iteration count is available
through preconditioning, which leaves the minimiser untouched. A symmetric positive definite
preconditioner $P_{\mathcal{F}} \approx A_{\mathcal{F}}$ turns the inner
solve~\eqref{eq:innerspd} into preconditioned CG (PCG), equivalently CG on the symmetrically
scaled operator $P_{\mathcal{F}}^{-1/2} A_{\mathcal{F}} P_{\mathcal{F}}^{-1/2}$, whose spectrum
coincides with that of $P_{\mathcal{F}}^{-1} A_{\mathcal{F}}$. PCG realises this without ever
forming $P_{\mathcal{F}}^{-1/2}$: each iteration adds a single application of
$P_{\mathcal{F}}^{-1}$ to the one mat-vec with $A_{\mathcal{F}}$~\cite{golub2013,benzi2005}.

\begin{proposition}[Preconditioned inner rate]\label{prop:pcg}
  Let $P_{\mathcal{F}}$ be SPD and write $\kappa_P = \kappa(P_{\mathcal{F}}^{-1}
  A_{\mathcal{F}})$. The PCG iterates for~\eqref{eq:innerspd} satisfy the energy-norm bound of
  Proposition~\ref{prop:cg} with $\kappa$ replaced by $\kappa_P$, so reaching a fixed relative
  tolerance costs $O(\sqrt{\kappa_P}\,\log\varepsilon^{-1})$ iterations. A preconditioner helps
  exactly when $\kappa_P < \kappa(A_{\mathcal{F}})$, and is ideal, $\kappa_P = 1$, when
  $P_{\mathcal{F}} = A_{\mathcal{F}}$.
\end{proposition}

\paragraph{Preconditioning on a changing free set.}
Two constraints decide which preconditioners suit this method, both inherited from
Section~\ref{sec:reduction}. First, the apply must be \emph{matrix-free}. Forming
$A_{\mathcal{F}}$ or a dense factor to build or invert $P_{\mathcal{F}}$ would forfeit the
$O(m\,|\mathcal{F}|)$ operator and the absence of $O(n^2)$ storage that motivate the whole
construction. Second, and particular to the active-set loop, the free set $\mathcal{F}$
changes between outer steps. Algorithm~\ref{alg:elim} rebuilds the operator on each reduced
$\mathcal{F}$, so a preconditioner is cheap only if it \emph{restricts} to the new free set at
negligible cost. The standard choices sort by exactly this property.

The Jacobi preconditioner $P = \operatorname{diag}(A)$ restricts exactly,
$(\operatorname{diag} A)_{\mathcal{F}} = \operatorname{diag}(A_{\mathcal{F}})$, so
$P_{\mathcal{F}}$ is the free-set slice of a single vector computed once. For $A = M^\top M$
that vector is the squared column norms of $M$, available without forming $A$, and its apply
is $O(|\mathcal{F}|)$. It carries no per-step setup and is valid on every free set the loop
visits, the natural default when $A$ has a widely varying diagonal.

A factorised preconditioner (incomplete Cholesky, $P = LL^\top$) does \emph{not} restrict. In
general $(P_{\mathcal{F}})^{-1} \neq (P^{-1})_{\mathcal{F}}$, and the factor of the full $P$ is
not a factor of $P_{\mathcal{F}}$, so honouring $A_{\mathcal{F}}$ would require refactorising
on each free set, reintroducing the assembly the method avoids. Such preconditioners pay off
here only once the free set stabilises, as it does near convergence and across the
support-stable warm starts of Proposition~\ref{prop:warm}, where one factor serves many outer
steps.

Between these sits the low-rank-plus-diagonal preconditioner
$P = \operatorname{diag}(d) + U\Delta U^\top$ built from a few ($r$) dominant eigenpairs of
$A$~\cite{saad2003iterative}. It restricts consistently in its \emph{apply},
$P_{\mathcal{F}} = \operatorname{diag}(d)_{\mathcal{F}} + U_{\mathcal{F}}\Delta
U_{\mathcal{F}}^\top$, and is inverted matrix-free in $O(|\mathcal{F}|\,r^2 + r^3)$ by the
Woodbury identity, precisely the direct free-block solve for the diagonal-plus-low-rank case.
It captures the leading spectral structure a diagonal misses. When $A$ is itself
diagonal-plus-low-rank the preconditioner \emph{is} $A$, giving $\kappa_P = 1$ and a direct
inner solve that bypasses CG altogether.

The \texttt{nncg} package realises this low-rank-plus-diagonal preconditioner by randomized
Nyström sketching, in a resketch-per-block variant (\texttt{Nystrom}) and a
sketch-once-and-mask variant (\texttt{GlobalNystrom}) that trade setup cost against reuse
exactly as the factorised-preconditioner case above; the global sketch amortises across the
support-stable warm-started re-solves of Proposition~\ref{prop:warm}. Appendix~\ref{app:precond}
tests both Jacobi and Nyström preconditioning and gives the Nyström construction.

\paragraph{Stopping criterion.}
PCG naturally monitors the preconditioned residual norm $(r^\top P_{\mathcal{F}}^{-1}
r)^{1/2}$, whereas the transfer of finite termination to inexact solves
(Lemma~\ref{lem:inexact}) is stated for the ordinary residual $\rho = \lVert b_{\mathcal{F}} -
A_{\mathcal{F}}\tilde x_{\mathcal{F}}\rVert_2$ that enters condition~\eqref{eq:inexactcond}.
The two differ by a factor between $\lambda_{\min}(P_{\mathcal{F}})^{1/2}$ and
$\lambda_{\max}(P_{\mathcal{F}})^{1/2}$. To keep the guarantee intact, the inner stop is
evaluated on the unpreconditioned $\rho$ (or the preconditioned tolerance is tightened by
$\kappa(P_{\mathcal{F}})^{1/2}$). Lemma~\ref{lem:inexact} then applies word for word, because
preconditioning changes only how fast $\rho$ falls, leaving $A$, the residual $\rho$ itself,
and the decision margin $\mu$ untouched.

Two remarks connect the subsection to the rest of the paper. First, the regularising
split~\eqref{eq:regsplit} may itself be read as an implicit preconditioner. It changes the
problem, but toward a target the application deems statistically or structurally preferable, so
the conditioning gain comes at no modelling cost while additionally repairing rank deficiency.
Second, preconditioning is a lever the Krylov inner solver has that the projected-gradient
reference of Section~\ref{ssec:proj} effectively lacks. A scaled gradient step minimises $f$ in
the $P$-metric, but the orthant and simplex projections that make projected gradient cheap are
closed-form only in the Euclidean norm, and become non-negative quadratic programs under a
general $P$. Preconditioning therefore \emph{widens} the $\sqrt\kappa$ advantage of the
active-set/CG loop over projected gradient rather than closing it.

%% file: sections/s7_results.tex
\section{Complexity and Numerical Behaviour}\label{sec:results}

\subsection{Complexity}\label{ssec:complexity}

Table~\ref{tab:complexity} collects the asymptotic cost of solving~\eqref{eq:bqp} by each
method, for $A = M^\top M$ with $M \in \mathbb{R}^{m\times n}$, condition number
$\kappa = \kappa(A_\alpha)$ after the regularising split, tolerance $\varepsilon$, free-set
size $k = |\mathcal{F}| \leq n$, and $s$ active-set outer steps. ``Extra memory'' is storage
beyond the operator; the matrix-free rows never form $A$ and need only $O(n)$ working memory.

\begin{table}[ht]
\centering
\footnotesize
\begin{tabular}{lcccc}
\toprule
Method & Per-step cost & Iterations & Extra memory & Forms $A$ \\
\midrule
Interior point (dense QP)          & $O(n^3)$        & $O(\log 1/\varepsilon)$          & $O(n^2)$ & Yes \\
Dense Cholesky $+$ active set       & $O(k^2 m)$      & $s$                             & $O(k^2)$ & Yes \\
CG matrix-free $+$ active set        & $O(km)$         & $s \cdot O(\sqrt\kappa)$        & $O(n)$   & No  \\
Factor / Woodbury $+$ active set     & $O(kr^2 + r^3)$ & $s$                             & $O(nr)$  & No  \\
Projected gradient (matrix-free)     & $O(nm)$         & $O(\kappa)$                     & $O(n)$   & No  \\
\bottomrule
\end{tabular}
\caption{Asymptotic cost of solving the non-negative quadratic~\eqref{eq:bqp} with
$A = M^\top M$, $M \in \mathbb{R}^{m \times n}$. Here $k = |\mathcal{F}| \leq n$ is the
free-set size, $s$ the number of active-set outer steps, $\kappa = \kappa(A_\alpha)$ the
condition number after the split~\eqref{eq:regsplit}, and $r$ the factor rank of a
diagonal-plus-low-rank operator. The last four rows share the active-set loop and differ only
in the backend: matrix-free CG and the Woodbury factor solve avoid assembling $A$, whereas
dense Cholesky forms and factorises $A_{\mathcal{F}}$. The CG ``Iterations'' column counts
$s \cdot O(\sqrt\kappa)$ inner iterations, dominated by the first solve at $k = n$; projected
gradient pays $O(\kappa)$ at the same per-step cost.}
\label{tab:complexity}
\end{table}

Two comparisons are structural. Within the active-set family, matrix-free CG and dense
Cholesky share the outer loop and differ only in whether $A_{\mathcal{F}}$ is assembled. The
$O(km)$-versus-$O(k^2 m)$ gap means CG pulls ahead as the free set grows, while dense Cholesky
can win when $k$ is small or $\kappa$ is large enough that CG needs many inner iterations
(the crossover is near $\kappa \sim k^2$). Between the two matrix-free methods, CG's
$O(\sqrt\kappa)$ dominates projected gradient's $O(\kappa)$ whenever $\kappa$ is large. Only
when the regularising split compresses $\kappa$ to $O(1)$ do the inner counts become
comparable, and then projected gradient's loop-free simplicity is attractive.

The equality-augmented problem~\eqref{eq:eqp} multiplies the inner cost of an outer step by the
$p+1$ right-hand sides that share the operator $A_{\mathcal{F}}$
(Section~\ref{sec:reduction}), plus an $O(p^3)$ dense Schur solve that is negligible for the
small $p$ of typical balance conditions. The asymptotics in $n$, $m$, and $\kappa$ are
otherwise those of Table~\ref{tab:complexity}.

\subsection{A controlled test family}\label{ssec:synthetic}

We test the propositions on a controlled family with a known optimum. Take
$A = Q\Lambda Q^\top$ with $Q$ a random orthogonal matrix and a geometric spectrum
$\Lambda = \mathrm{diag}(\lambda_1,\dots,\lambda_n)$ on $[1,\kappa]$ of prescribed condition
number $\kappa$ (equivalently $A = M^\top M$ for a random $M$ with the chosen singular values).
Plant a non-negative solution $x^\star \geq 0$ of a chosen support by setting
$b = A x^\star - s^\star$ for a complementary $s^\star \geq 0$. Then $(x^\star, s^\star)$ is
the exact solution of $\mathrm{LCP}(A,\,{-b})$ in closed form. The runs below use $n=200$
(support $50\%$) averaged over five seeds, with the smaller $n=120$ where the projected-gradient
$O(\kappa)$ cost forces a capped $\kappa$ range. All are reproduced by a self-contained
NumPy script that implements Algorithm~\ref{alg:elim}, the CG inner solve, and the reference
method directly from their definitions. The matrix-free deblurring study is likewise
NumPy-only. The external-solver benchmark below and the \texttt{shaw}/\texttt{phillips}
ill-posed study are two further scripts whose only additional dependency is SciPy (and, for the
benchmark, Clarabel). The warm-started efficient-frontier study reads the committed S\&P~500
returns panel and adds pandas. The algorithms are also released as the
open-source package \texttt{nncg} (\url{https://github.com/Jebel-Quant/nncg}), whose test
suite is this study. Planted recovery, the equality-augmented solve, trajectory agreement
under inexact solves, the warm-started frontier's single-step property, the
\texttt{shaw}/\texttt{phillips} regularisation and recovery runs, the matrix-free deblurring
solve, and the adversarial family that forces the fallback all run as continuous-integration
checks. The geometric spectrum is
the honest choice: its clustering lets CG converge faster than the worst-case bound, so the
measured growth is an \emph{upper}-bounded confirmation of Proposition~\ref{prop:cg}, not a
fitted exponent.

\begin{table}[ht]
\centering
\footnotesize
\input{tables/nncg_synthetic}
\caption{Active-set solve of~\eqref{eq:bqp} on the synthetic family ($n=200$, support $50\%$,
mean over five seeds). Outer steps stay in single digits, and the total CG inner count grows
sublinearly in $\kappa$ (slope $\nncgCGslope$ in $\log$--$\log$, $R^2 = \nncgCGrtwo$, against
the $O(\sqrt\kappa)$ worst case). The Bland fallback is never triggered here, so the fast
block-pivot path certifies optimality alone. The solver recovers the planted $x^\star$ to
$\nncgMaxErr$.}
\label{tab:nncg_synthetic}
\end{table}

Three previous observations are now confirmed.

\emph{(i) Inner count is bounded by $\sqrt\kappa$; regularisation lowers it.}
Figure~\ref{fig:nncg_kappa} shows the total CG inner iterations rising with $\kappa$ but
staying under the $O(\sqrt\kappa)$ envelope of Proposition~\ref{prop:cg} (measured slope
$\nncgCGslope < \tfrac12$, the clustered spectrum buying CG its superlinear phase). The
regularising split~\eqref{eq:regsplit} with $R^\top R = I$ compresses $\kappa$ at rate
$O((1-\alpha)/\alpha)$ (Proposition~\ref{prop:regcond}) and drives the count down monotonically
as $\alpha$ increases (Figure~\ref{fig:nncg_reg}). The effect is largest in the rank-deficient
regime $m < n$, where only $\alpha > 0$ makes the system well posed
(Table~\ref{tab:nncg_rankdef} below).

\emph{(ii) Outer count is small and the fallback is dormant.}
Across $\kappa \in [10, 10^6]$ the active-set loop terminates in at most $\nncgMaxOuter$ outer
steps (Table~\ref{tab:nncg_synthetic}), adjusting the free set by a few indices per step and
reaching the correct support in a count that tracks its distance from the initial full set, not
$n$. The patience budget is never exhausted: the Bland fallback of
Theorem~\ref{thm:termination} guarantees termination but is not entered on generic data, exactly
the codimension-one argument of Section~\ref{sec:nonneg}.

\emph{(iii) CG beats projected gradient by a widening $\sqrt\kappa$ factor.}
Figure~\ref{fig:nncg_cgpg} plots both matrix-free methods at matched per-step cost. Both run
faster than their worst-case bounds on the clustered spectrum, but projected gradient's growth
exponent is about twice CG's, so their ratio grows as $\kappa^{\nncgRatioSlope}$ (essentially
the predicted $\sqrt\kappa$) from parity at $\kappa=10$ to $\nncgPGratio\times$ at
$\kappa = \nncgKappaMax$. This is the $O(\sqrt\kappa)$-versus-$O(\kappa)$ separation of
Propositions~\ref{prop:cg} and~\ref{prop:proj}, widening with $\kappa$ and closing under heavy
regularisation.

\begin{figure}[ht]
\centering
\begin{minipage}{0.32\textwidth}
  \centering
  \includegraphics[width=\linewidth]{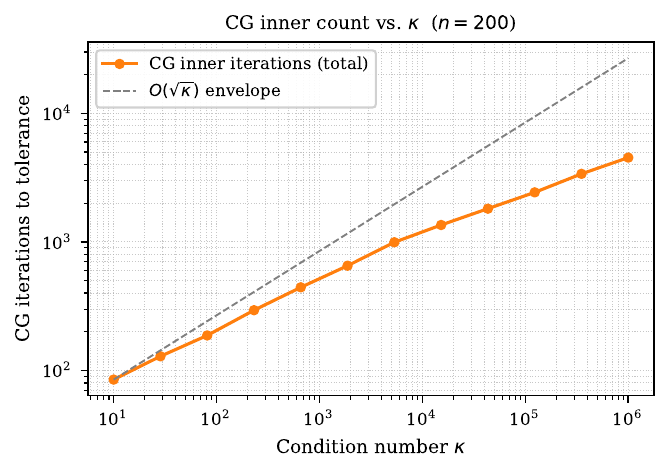}
  \caption{CG inner count vs.\ $\kappa$, under the $O(\sqrt\kappa)$ envelope.}
  \label{fig:nncg_kappa}
\end{minipage}\hfill
\begin{minipage}{0.32\textwidth}
  \centering
  \includegraphics[width=\linewidth]{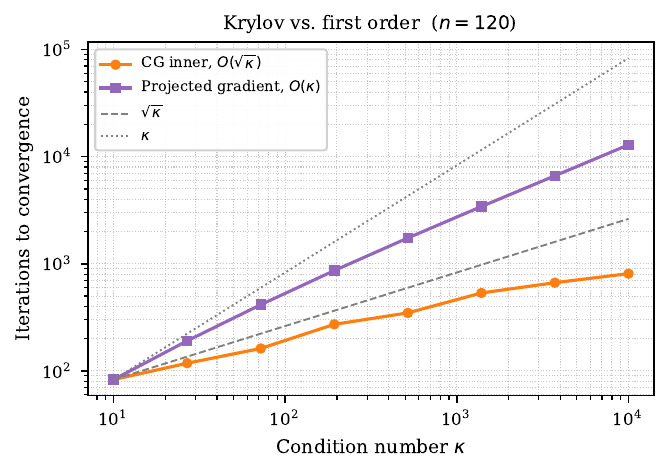}
  \caption{Krylov vs.\ first order: $O(\sqrt\kappa)$ vs.\ $O(\kappa)$.}
  \label{fig:nncg_cgpg}
\end{minipage}\hfill
\begin{minipage}{0.32\textwidth}
  \centering
  \includegraphics[width=\linewidth]{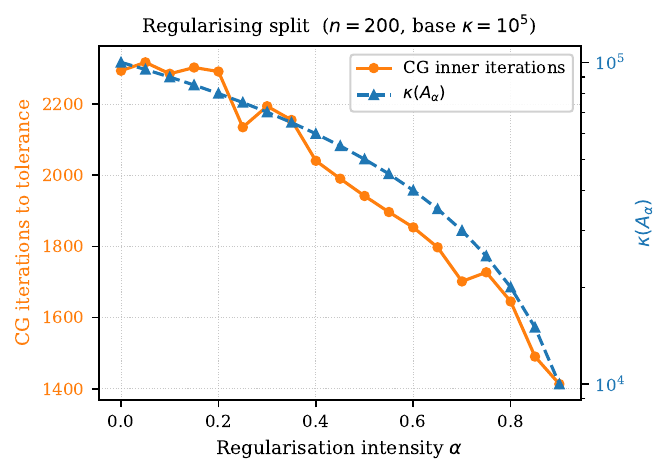}
  \caption{Regularising split: count falls as $\kappa(A_\alpha)$ shrinks.}
  \label{fig:nncg_reg}
\end{minipage}
\end{figure}

\subsubsection{Exercising the fallback}\label{sssec:fallback}

That the fallback lies dormant on generic data raises the
question of whether it is needed at all. The following family answers it. Let the columns of
$M$ arrive in near-anti-parallel pairs, $M = [\,M_0,\ -M_0 + \xi E\,]$, with small $\xi,$ random perturbation $E$, and a
ridge to keep $A = M^\top M + \lambda I$ a $P$-matrix. Pushing a variable to the bound flips
the sign of its partner's entry, so a batch drop systematically over-shoots and a later batch
add restores it. This is the codimension-one near-tie event of Section~\ref{sec:nonneg} made
systematic. On this family ($n = 20$, \nncgFbSeeds{} seeds), \emph{pure} block principal
pivoting (Algorithm~\ref{alg:elim} with unbounded patience and exact inner solves)
revisits a previously seen free set and cycles without terminating on \nncgFbCycles{} of
\nncgFbSeeds{} seeds. The guarded loop, exactly as stated (default patience, CG inner
solves), converges on all \nncgFbSeeds{}. On \nncgFbFired{} seeds the patience budget is
exhausted and up to \nncgFbMax{} least-index pivots are taken before the batch path resumes,
each exit certified by the KKT system to working tolerance. The guarantee of
Theorem~\ref{thm:termination} is therefore not decorative: on adversarial data the unguarded
fast path genuinely cycles, and the fallback is what terminates it.

\subsubsection{Rank deficiency}

Table~\ref{tab:nncg_rankdef} probes the semidefinite regime the
regularising split is claimed to repair: a Gram operator $A = M^\top M$ with
$M \in \mathbb{R}^{100\times 200}$. Every free block with $|\mathcal{F}| > m = 100$ is then
singular, including the initial one, since the loop starts from
$\mathcal{F} = \{1,\dots,n\}$. The failure at $\alpha = 0$ is instructive in its anatomy.
When the planted support lies \emph{below} the rank, the first inner solve can never meet its
tolerance: the residual stalls at the component of $b_{\mathcal{F}}$ outside the range of
$A_{\mathcal{F}}$, so the iteration cap is hit and the iterate carries no certificate. Yet
the uncertified signs may still identify the support, after which the shrunken free blocks
are regular. The loop then reaches the optimum on \nncgRankConvLow{} of five seeds (at an
order of magnitude more inner iterations) and fails to converge on the remainder. When the
optimal support \emph{exceeds} the rank, the final free block is itself singular, no
certificate is ever available, and the loop converges on \nncgRankConvHigh{} of five seeds.
Any $\alpha > 0$ restores the $P$-matrix premise of Theorem~\ref{thm:termination}, and with
it both reliability and speed: every run converges, recovering the planted optimum to
$\nncgRankOkErr$. The partial successes at $\alpha = 0$ sharpen the point rather than soften
it. What regularisation buys is not simply conditioning but the \emph{guarantee} as without it termination of the algorithm is luck.

\begin{table}[ht]
\centering
\footnotesize
\input{tables/nncg_rankdef}
\caption{The rank-deficient regime ($n = 200$, $m = 100$, five seeds; CG capped at $2000$
iterations per solve, $30$ outer steps). ``Capped solves'' counts runs whose first inner
solve hit the cap without meeting its tolerance. At $\alpha = 0$ the loop is unreliable below
the rank and fails uniformly above it; any $\alpha > 0$ converges on every seed with an order
of magnitude fewer inner iterations.}
\label{tab:nncg_rankdef}
\end{table}

\subsubsection{A recognised ill-posed problem}

The synthetic family is planted for control; to confirm
the same behaviour on a recognised instance, we take the \texttt{shaw} problem from Hansen's
Regularization Tools~\cite{hansen1994regutools,shaw1972}, a severely ill-posed one-dimensional
image-restoration model. Its exact solution, a sum of two Gaussians, is strictly positive and
hence a genuine non-negative least-squares target. Its symmetric kernel $M$ is numerically
rank-deficient, so the Gram operator $A = M^\top M$ ($n = \nncgReguN{}$) is numerically
singular, $\kappa(A) \approx \nncgReguKappa{}$. Adding $\nncgReguNoise{}$ relative noise to the
data makes the unconstrained least-squares solution oscillate negative, so non-negativity
binds as an active regulariser. Table~\ref{tab:nncg_regu} repeats the split sweep of
Table~\ref{tab:nncg_rankdef} on this operator. At $\alpha = 0$ the loop cannot certify a
solution: the singular free block leaves CG stalling on the inconsistent system, the outer
budget is exhausted, and the KKT residual stays at $O(10)$. Any $\alpha > 0$ restores the
$P$-matrix property, and the loop terminates in a handful of outer steps with a KKT
certificate. At the representative $\alpha = \nncgReguAlpha{}$ (which lowers $\kappa$ to
$\nncgReguKappaReg{}$) it converges in $\nncgReguOuter{}$ outer steps with $\nncgReguActive{}$
constraints active, agreeing with SciPy's Lawson--Hanson on the same regularised operator to
$\nncgReguErrLH{}$. The speed advantage over Lawson--Hanson is the one the synthetic benchmark
already records. The block loop reaches the $\nncgReguSupp{}$-nonzero solution in
$\nncgReguOuter{}$ outer steps, whereas Lawson--Hanson admits a single index per pivot and so
needs on the order of $\nncgReguSupp{}$ of them. On the regularised operator this is a
$\nncgBenchShawLHvsASCG{}\times$ wall-clock speedup at $n = \nncgBenchShawN{}$, at an identical
solution. The official problem tells the planted one's story: without the split,
termination is luck; with it, the guarantee holds. A second Regularization Tools
problem, \texttt{phillips}~\cite{hansen1994regutools,phillips1962} ($n = \nncgPhilN{}$), has a
known non-negative solution that is exactly zero outside a central interval, so the active set
is large and structured. The loop recovers the signal, certifies KKT optimality, and pins
$\nncgPhilActive{}$ variables at the bound against the $\nncgPhilTrueZeros{}$ true zeros,
agreeing with Lawson--Hanson to $\nncgPhilErrLH{}$.

\begin{table}[ht]
\centering
\footnotesize
\input{tables/nncg_regu}
\caption{The \texttt{shaw} ill-posed problem from Hansen's Regularization
Tools~\cite{hansen1994regutools} ($n = \nncgReguN{}$, $\nncgReguNoise{}$ relative data noise;
CG capped at $2000$ iterations per solve, $30$ outer steps). ``Active'' counts components
pinned at the bound. At $\alpha = 0$ the singular Gram operator
($\kappa \approx \nncgReguKappa{}$) leaves the loop uncertified; any $\alpha > 0$ terminates
with a KKT certificate, matching Lawson--Hanson on the regularised operator.}
\label{tab:nncg_regu}
\end{table}

\subsection{Matrix-free operation at scale}\label{ssec:matrixfree}

The problems so far are small enough to store $A$ densely. The
method's reason for existing is the regime where that is impossible. A separable Gaussian blur
$B = K \otimes K$ on an $N \times N$ image, applied as $X \mapsto K X K^\top$, gives a
non-negative deconvolution~\cite{hansen2006deblurring} whose Gram operator $A = B^\top B$ has
$n = N^2$ unknowns. The matrix-free loop (the operator solver of Section~\ref{sec:reduction},
reaching $A$ only through $v \mapsto B^\top(Bv)$) never forms it. At $N = \nncgDeblurNside{}$,
so $n = \nncgDeblurDim{}$, a dense $A$ would occupy $\nncgDeblurDenseGB{}$\,GB. With a ridge
split for the $P$-matrix property, the loop instead runs in $O(n)$ working memory, terminating
in $\nncgDeblurOuter{}$ outer steps ($\nncgDeblurInner{}$ CG iterations total) with a KKT
certificate and $\nncgDeblurActive{}$ of $\nncgDeblurDim{}$ pixels pinned at the non-negativity
bound (Figure~\ref{fig:nncg_deblur}). Identical loop, identical proof; only the operator
changes.

\begin{figure}[ht]
\centering
\includegraphics[width=\linewidth]{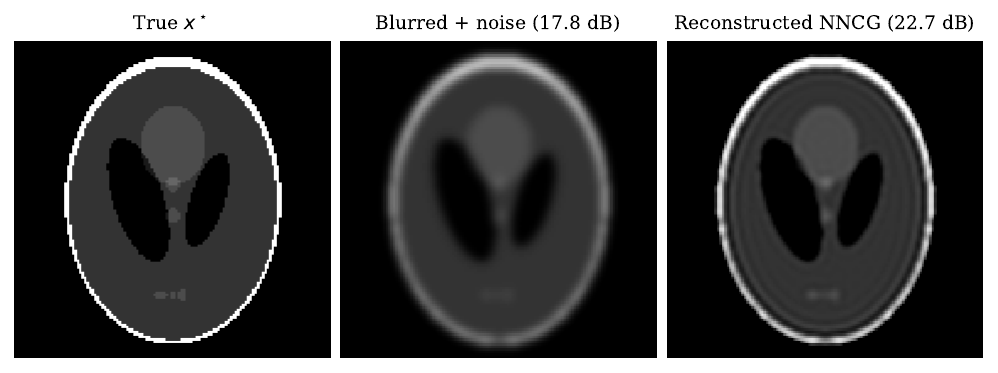}
\caption{Matrix-free non-negative deblurring of the Shepp--Logan phantom at
$N = \nncgDeblurNside{}$ ($n = \nncgDeblurDim{}$ unknowns): the true image, its Gaussian blur
with additive noise, and the reconstruction from the matrix-free active-set loop
(Algorithm~\ref{alg:elim}), which reaches $A = B^\top B$ only through $v \mapsto B^\top(Bv)$ and
never forms it, a dense $A$ would be $\nncgDeblurDenseGB{}$\,GB. The solve lifts the image from
$\nncgDeblurPSNRblur{}$ to $\nncgDeblurPSNR{}$\,dB PSNR and pins the $\nncgDeblurActive{}$
black-background pixels exactly at the non-negativity bound.}
\label{fig:nncg_deblur}
\end{figure}

\subsection{Inexactness and general equality constraints}\label{ssec:robustness}


In line with Lemma~\ref{lem:inexact}, running the loop with its
CG inner solver ($\varepsilon_{\mathrm{cg}} = 10^{-10}$, $\varepsilon = 10^{-8}$) against a
twin loop whose inner solver is an exact direct solve produces \emph{identical} free-set
trajectories on all \nncgTrajTotal{} planted instances
($\kappa \in \{10^2, 10^4, 10^6\}$, five seeds each). The guarantee of
Theorem~\ref{thm:termination} survives the inexact solves it is implemented with, not merely
in principle but decision-for-decision.


The same loop with the Schur-complement elimination of
Section~\ref{sec:reduction} solves the equality-augmented problem~\eqref{eq:eqp} for a general
$Bx = c$. On planted instances with $B \in \mathbb{R}^{p\times n}$ of full row rank, the solver
recovers the optimum to $\nncgEqErr$ and meets the equality to machine precision for
$p \in \{1, 3, 10\}$, with the inner cost scaling as the $p+1$ shared right-hand sides
anticipated in Section~\ref{sec:reduction}. The case $p=1$ reproduces the single-normalisation
solve exactly.

\subsection{Warm-starting a parametric sweep}\label{ssec:frontier}


The parametric sweep a practitioner
actually runs is the long-only efficient frontier itself. On the S\&P~500 (\nncgFrontN{}
assets, daily returns \nncgFrontStart{}--\nncgFrontEnd{}) we fit a rank-\nncgFrontK{} factor
covariance $\Sigma = D + L L^\top$---the diagonal-plus-low-rank operator of
Table~\ref{tab:complexity}, whose \nncgFrontK{} statistical factors carry
\nncgFrontSysFrac{}\% of the variance---and trace the frontier
$\min_{w\ge0,\,\mathbf{1}^\top w = 1}\tfrac12 w^\top\Sigma w - \gamma\,\mu^\top w$ over
\nncgFrontGammas{} values of the risk-aversion $\gamma$, the equality-augmented problem of
Section~\ref{sec:reduction} with $B = \mathbf{1}^\top$. The operator is reached only through
$v \mapsto D v + L(L^\top v)$, so $\Sigma$ is never formed and the matrix-free CG inner solve
of Proposition~\ref{prop:cg} is the natural one. As $\gamma$ grows the holding concentrates
from \nncgFrontHeldMax{} names toward \nncgFrontHeldMin{}, so adjacent frontier points share
most of their support. Warm-starting each solve from the previous $\gamma$'s $(\mathcal{F}, x)$
pair cuts the mean outer count from \nncgWarmColdOuter{} to \nncgWarmOuter{} and the total CG
inner count by a factor of \nncgWarmInnerX{} (Figure~\ref{fig:nncg_frontier}), a
\nncgWarmX{}$\times$ wall-clock speedup at an identical optimum (max deviation \nncgWarmErr{}).
The support-stable case of Proposition~\ref{prop:warm} is confirmed on live data: of the
\nncgWarmStableTot{} steps, \nncgWarmStable{} leave the optimal support unchanged, and
\emph{every one} is solved in a single warm-started outer step. The gain is a property of the
outer loop, so it is independent of the inner solver---the same reduction attends the direct
free-set solve. MPRGP does not enter this comparison at all: the budget $\mathbf{1}^\top w = 1$
is an equality constraint outside its bound-only scope, reachable only by wrapping it in an
outer augmented-Lagrangian loop~\cite{dostal2006smalbe}, whereas the active-set loop absorbs it
directly through the Schur-complement elimination of Section~\ref{sec:reduction}. The long-only
frontier is thus not a problem on which the two methods compete---it is one the feasible-point
family cannot pose.

\begin{figure}[ht]
\centering
\includegraphics[width=\textwidth]{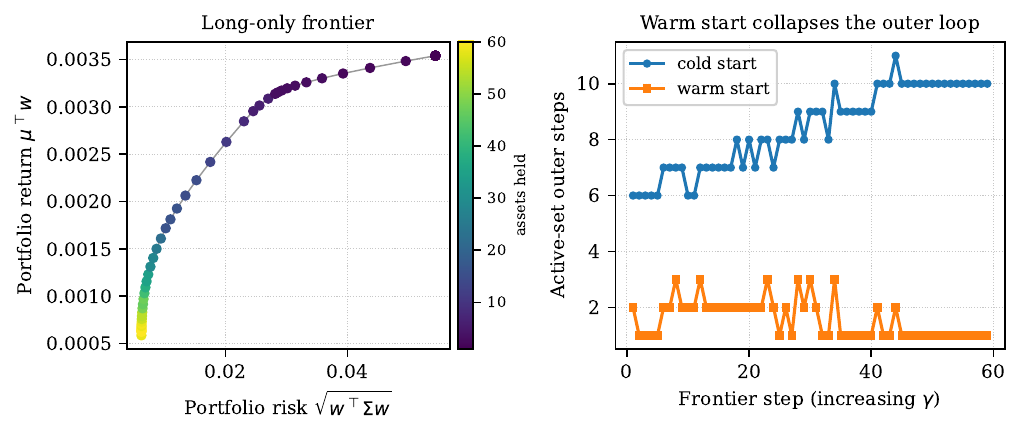}
\caption{Warm-starting the long-only efficient frontier on a matrix-free factor covariance
$\Sigma = D + L L^\top$ (S\&P~500, $n = \nncgFrontN{}$). \emph{Left}: the frontier, each
$\gamma$ coloured by the number of names held, concentrating from the min-variance support
toward a single asset. \emph{Right}: active-set outer steps per frontier point, cold versus
warm-started from the previous $\gamma$. The reduction is a property of the outer loop, hence
independent of the inner solver; the warm run pays only for the support drift
(Proposition~\ref{prop:warm}).}
\label{fig:nncg_frontier}
\end{figure}


Because the factor operator is never assembled, the
sweep runs in $O(nK)$ memory where a dense $\Sigma$ cannot be formed. Calibrating the fitted
model to a synthetic universe of $n = \nncgFrontMaxN{}$ assets---a dense $\Sigma$ of
\nncgFrontMaxGB{}\,GB---the warm-started sweep traces the entire frontier in under a second, a
\nncgFrontMaxX{}$\times$ speedup over the cold sweep (Table~\ref{tab:nncg_frontier}), in the
regime where the assembled solvers of the external benchmark cannot instantiate the operator at
all. The efficient frontier is one representative parametric family; the same warm-started loop
applies to any slowly varying sequence of non-negative quadratic programs---rolling
rebalancing, regularisation-path sweeps, scenario families---which we do not pursue further
here.

\begin{table}[ht]
\centering
\footnotesize
\input{tables/nncg_frontier}
\caption{Warm-started matrix-free frontier on a calibrated factor universe: cold versus
warm-started wall-clock (s) for the whole \nncgFrontGammas{}-point sweep against problem size
$n$, with the storage a dense $\Sigma$ would need but the factor operator never forms. Applied
in $O(nK)$; the direct solvers of the external benchmark cannot factorise $\Sigma$ at these
sizes.}
\label{tab:nncg_frontier}
\end{table}

\subsection{Comparison with external solvers}\label{ssec:external}
\paragraph{Dense, assembled problems}

The comparisons so far are internal, i.e. active-set variants against
one another and against projected gradient. We close the section by placing
Algorithm~\ref{alg:elim} among the external solvers a practitioner would actually reach for.
This dense, explicitly assembled family is chosen to make that placement fair to the
baselines, not to exhibit the matrix-free regime the method is built for. On it the same
planted family is solved by SciPy's Lawson--Hanson NNLS~\cite{lawson1974} (fed the Cholesky
factor $M = L^\top$, factorisation included in the timing) and the interior-point solver
Clarabel~\cite{clarabel2024}, against Algorithm~\ref{alg:elim} with a direct and with the
matrix-free CG inner solve. Figure~\ref{fig:nncg_bench} reports wall-clock times against $n$
at $\kappa = 10^4$; Table~\ref{tab:nncg_bench} reports the $\kappa$-sensitivity at $n = 500$
and the accuracy of every solver against the planted optimum. Three observations follow.
First, the active-set loop dominates both baselines on this family, at $n = \nncgBenchRatioN$
by a factor of \nncgBenchLHvsASD{} over Lawson--Hanson and \nncgBenchIPvsASD{} over Clarabel
(direct inner solve). The reason is structural (Section~\ref{sec:nonneg}): block pivoting
reaches the optimal support in single-digit outer steps, where Lawson--Hanson moves one index
per iteration and the interior point pays its Newton solves on the full $n \times n$ system.
Second, the two inner solvers answer different questions, and this benchmark is the one that
favours the direct block solve. It is faster in single digits at $n = 500$ (dense BLAS-3
factorisations beat memory-bound mat-vecs at these sizes) and is $\kappa$-insensitive, whereas
the CG inner cost grows as $\sqrt\kappa$ (Table~\ref{tab:nncg_bench}: a factor
\nncgBenchASCGSens{} from $\kappa = 10^2$ to $10^6$). This is by design, not a deficit. Where
$A$ is small and dense enough to assemble and factorise, the direct solve should win, and the
active-set loop (the paper's actual object) accommodates it unchanged: Theorem~\ref{thm:termination}
governs the exact inner solve, and Lemma~\ref{lem:inexact} carries the guarantee to the
inexact CG one. The matrix-free CG inner solve earns its $\sqrt\kappa$ rate and $O(n)$ working
memory only in the complementary regime, namely the Gram and structured operators of
Section~\ref{sec:reduction} that are never assembled, and warm-started parametric sweeps where
each solve reuses the previous free set. Third, every solver recovers the planted optimum:
Lawson--Hanson and the direct active set to $10^{-13}$, CG to its residual tolerance
($10^{-8}$), and Clarabel to its default interior-point tolerance ($10^{-4}$). The active-set
methods certify KKT optimality exactly rather than to a duality-gap tolerance.

\begin{figure}[ht]
\centering
\includegraphics[width=0.55\textwidth]{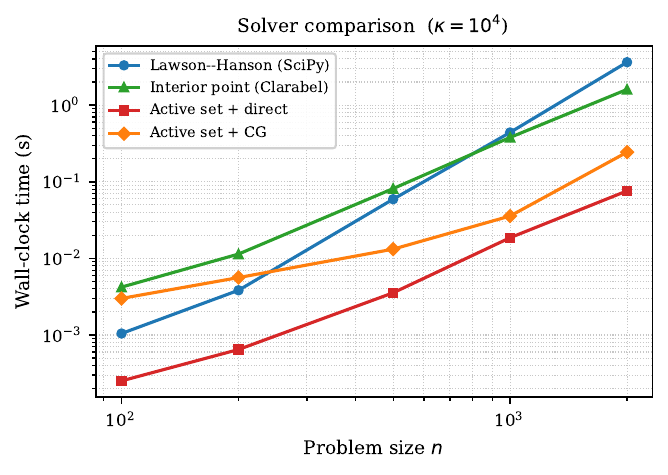}
\caption{Wall-clock time against problem size $n$ on the planted family
($\kappa = 10^4$, support $50\%$, mean over seeds): SciPy's Lawson--Hanson NNLS and the
Clarabel interior-point solver against Algorithm~\ref{alg:elim} with direct and matrix-free
CG inner solves.}
\label{fig:nncg_bench}
\end{figure}

\begin{table}[ht]
\centering
\footnotesize
\input{tables/nncg_bench}
\caption{$\kappa$-sensitivity at $n = 500$ (mean over three seeds) and accuracy against the
planted optimum. The direct methods are $\kappa$-insensitive at fixed $n$; only the CG inner
cost grows, as Proposition~\ref{prop:cg} prescribes. Accuracy reflects each solver's own
stopping criterion: the active-set exits are KKT certificates, the interior point stops at its
duality-gap tolerance.}
\label{tab:nncg_bench}
\end{table}

\paragraph{Matrix-free at scale}

The matrix-free deblurring family of Section~\ref{ssec:matrixfree} places the method against
four external solvers, split by what each needs of the operator. Lawson--Hanson, Clarabel, and
OSQP~\cite{osqp} need the explicit Gram operator $A = B^\top B$, available here in
closed form as $A_\alpha = (1-\alpha)(G \otimes G) + \alpha I$ with $G = K^\top K$, so they pay
$O(n^2)$ storage and an $O(n^3)$ factorisation. MPRGP~\cite{dostalschoberl2005}, the
feasible-point competitor of Section~\ref{sec:nonneg}, is matrix-free like the active-set loop
and pays neither. Table~\ref{tab:nncg_deblur} and Figure~\ref{fig:nncg_deblur_bench} report
wall-clock against $n$. At $n = \nncgDeblurHeadDim{}$ ($N = \nncgDeblurHeadNside{}$), the
largest size the dense solvers still handle, the active-set loop finishes in
$\nncgDeblurHeadMFTime{}$\,s against $\nncgDeblurHeadLHTime{}$\,s for Lawson--Hanson and
$\nncgDeblurHeadIPTime{}$\,s for Clarabel, a $\nncgDeblurHeadLHx{}\times$ and
$\nncgDeblurHeadIPx{}\times$ margin. Past that size the dense operator is the binding
constraint: at $N = \nncgDeblurMaxNside{}$ ($n = \nncgDeblurMaxDim{}$) it would occupy
$\nncgDeblurMaxDenseGB{}$\,GB, so none of the three dense solvers can even be instantiated,
whereas the active-set loop solves the problem in $\nncgDeblurMaxTime{}$\,s in $O(n)$ working
memory. This is the regime the method is built for. The win is not a constant factor on
assembled problems but a solver that runs at all once $A$ cannot be formed. The comparison also
makes the enabling property explicit rather than proprietary: MPRGP is matrix-free too, clears
the same wall, and on this well-conditioned blur is in fact the faster of the two ($O(n)$
memory, $\nncgDeblurMaxMPRGPTime{}$\,s at the largest size). What the family isolates is
matrix-free operation, which the active-set loop shares; the choice between it and MPRGP then
turns on the properties of Section~\ref{sec:nonneg} (an exact KKT certificate at termination,
and a free-set structure that carries a direct inner solve and warm starts) rather than on raw
speed here.

\begin{figure}[ht]
\centering
\includegraphics[width=0.55\textwidth]{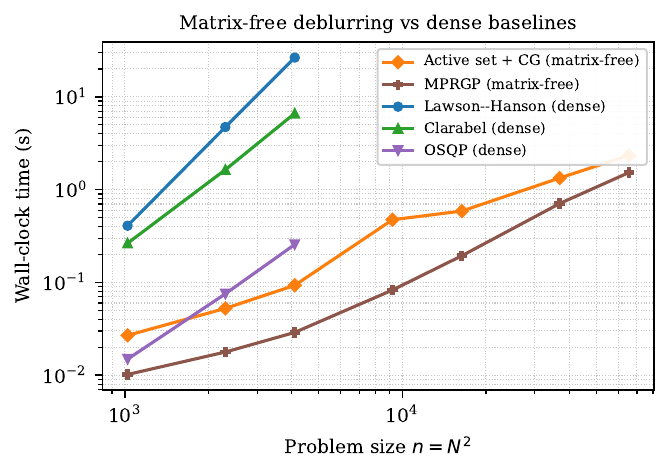}
\caption{Wall-clock time against problem size $n = N^2$ on the matrix-free deblurring family:
the two matrix-free methods (the active-set CG loop and MPRGP) against dense Lawson--Hanson,
Clarabel, and OSQP. The dense baselines stop where the Gram matrix $A = B^\top B$ ceases to be
practical to form and factorise; the matrix-free methods continue to $n = \nncgDeblurMaxDim{}$
(a dense $A$ would be $\nncgDeblurMaxDenseGB{}$\,GB) in $O(n)$ working memory. Both matrix-free
methods run unpreconditioned.}
\label{fig:nncg_deblur_bench}
\end{figure}

\begin{table}[ht]
\centering
\footnotesize
\input{tables/nncg_deblur}
\caption{Matrix-free deblurring: wall-clock (s) against problem size $n = N^2$, the two
matrix-free methods (MPRGP and the active-set CG loop) versus the dense Lawson--Hanson,
Clarabel, and OSQP (Gaussian blur $\sigma = 2$, ridge $\alpha = 10^{-3}$). ``dense $A$ (GB)''
is the storage the dense solvers would need for the Gram operator, which the matrix-free
methods never form; ``---'' marks sizes where that operator exceeds the benchmark's storage
budget and the dense solvers are not run. The active-set CG loop is run without a
preconditioner, as is MPRGP.}
\label{tab:nncg_deblur}
\end{table}

\paragraph{What a feasible-point method cannot certify}

The timing above conceded that MPRGP can be the faster matrix-free solver on a well-conditioned
operator. Speed, however, is not what separates the two methods. The active-set loop returns an
\emph{exact} KKT certificate, it identifies the optimal active set combinatorially and returns
a dual-feasible, complementary $(x, s)$ pair to machine precision, whereas MPRGP, a
feasible-point method, only drives the KKT residual below a requested tolerance and never
certifies the active set. On near-degenerate problems the gap is stark. We plant a
bound-constrained quadratic ($n = \nncgCertN{}$, $\nncgCertSeeds{}$ seeds) whose optimum
carries a geometric ladder of primal and dual margins down to $\nncgCertMinMargin{}$, so
identifying the support means resolving ever-smaller gaps. The active-set loop with an exact
inner solve recovers the support exactly---zero misclassifications---with a KKT residual of
$\nncgCertASresid{}$ in $\nncgCertASouter{}$ outer steps. MPRGP, run to a practical tolerance of
$\nncgCertMPtolHi{}$, mislabels $\nncgCertMPerrHi{}$ of the $\nncgCertN{}$ variables and exits
dual-infeasible (most-negative reduced gradient $\nncgCertMPdualHi{}$); its support error and
dual infeasibility fall only as fast as the tolerance it is asked for
(Figure~\ref{fig:nncg_certificate}), matching the active-set support only near
$\nncgCertMPtolZero{}$ ($\nncgCertMPitZero{}$ iterations), and even there it emits no
certificate. This is the structural payoff of the complementary-basis loop, compounded by the
equality constraint of Section~\ref{ssec:frontier}: the budget $\mathbf{1}^\top x = c$ is
outside MPRGP's bound-only scope entirely.

\begin{figure}[ht]
\centering
\includegraphics[width=\textwidth]{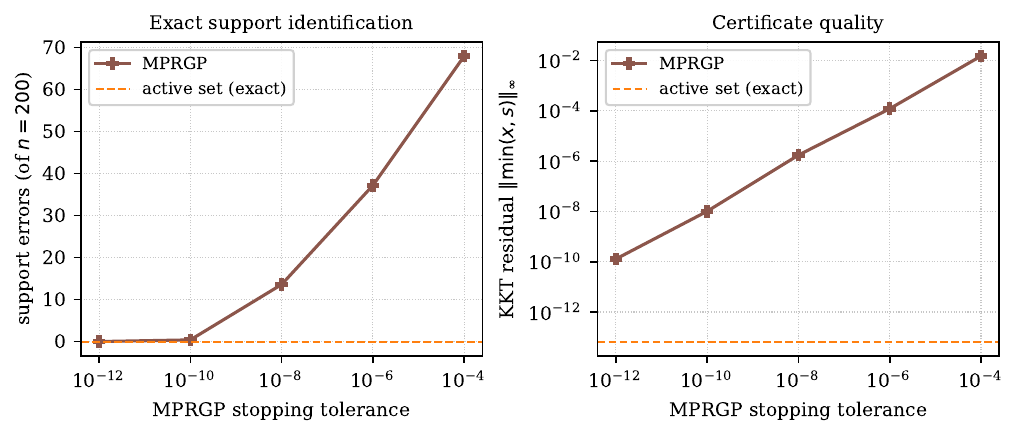}
\caption{Exact certificate versus feasible-point tolerance on a near-degenerate planted
bound-constrained QP ($n = \nncgCertN{}$, margins down to $\nncgCertMinMargin{}$,
$\nncgCertSeeds{}$ seeds). \emph{Left}: support-identification errors against the planted
optimum. \emph{Right}: the KKT natural residual $\|\min(x,s)\|_\infty$. The active-set loop
(dashed) is exact, zero support errors, residual $\nncgCertASresid{}$, in a fixed
$\nncgCertASouter{}$ outer steps; MPRGP's error and residual track its stopping tolerance, and
it never returns a dual-feasible certificate.}
\label{fig:nncg_certificate}
\end{figure}

\paragraph{A real, non-synthetic check at scale}

The external-baseline comparison above is on a planted family designed to be fair to
Lawson--Hanson and Clarabel alike; a genuine, over-complete hyperspectral unmixing problem
confirms the same ordering on real data at a scale where the matrix-free argument bites. We
unmix the Cuprite scene against the full USGS spectral library~\cite{clark2007},the standard
over-complete dictionary of the sparse-unmixing literature~\cite{iordache2011}, so each pixel
poses the fully-constrained
least-squares abundance problem $\min_x \tfrac12 x^\top A x - b_i^\top x$ s.t.\ $x \geq 0$,
$\mathbf{1}^\top x = 1$, $A = M^\top M$, now with $M$ holding $n = \nncgHsLibSignatures{}$
candidate mineral signatures over $\nncgHsLibBands{}$ bands. The Gram operator is
$\nncgHsLibSignatures{} \times \nncgHsLibSignatures{}$ of rank only $\nncgHsLibRank{}$, hence
rank-deficient; a ridge of $\nncgHsLibRidgeFrac{}$ restores the $P$-matrix premise
(Section~\ref{sec:conditioning}). At this $n$ Clarabel's per-pixel $O(n^3)$ factorisation is
the binding cost, exactly the regime Table~\ref{tab:complexity} anticipates and the one the
$n = 6$ Urban scene is far too small to exercise, so the active-set loop's matrix-free $O(km)$
step wins; FISTA is $\nncgHsLibSpeedupFista{}\times$ slower. Because neighbouring pixels share
their active dictionary members, raster-order warm starting cuts total outer steps from
$\nncgHsLibOuterCold{}$ to $\nncgHsLibOuterWarm{}$ ($\nncgHsLibSpeedupOuter{}\times$), a
$\nncgHsLibSpeedupWall{}\times$ wall-clock speedup, with $\nncgHsLibFracSingle{}\%$ of pixels
certified in a single warm-started outer step and cold and warm agreeing to $\nncgHsLibAgree{}$.
Most tellingly for the certificate, from an over-complete dictionary of
$\nncgHsLibSignatures{}$ candidates the loop identifies a mean of $\nncgHsLibSupportMean{}$
active signature per pixel (at most $\nncgHsLibSupportMax{}$): it recovers the exact,
single-mineral support with a KKT certificate, where a feasible-point method returns only an
approximate abundance vector left to be thresholded
(Figure~\ref{fig:nncg_hyperspectral_library}).

\begin{figure}[ht]
\centering
\includegraphics[width=\textwidth]{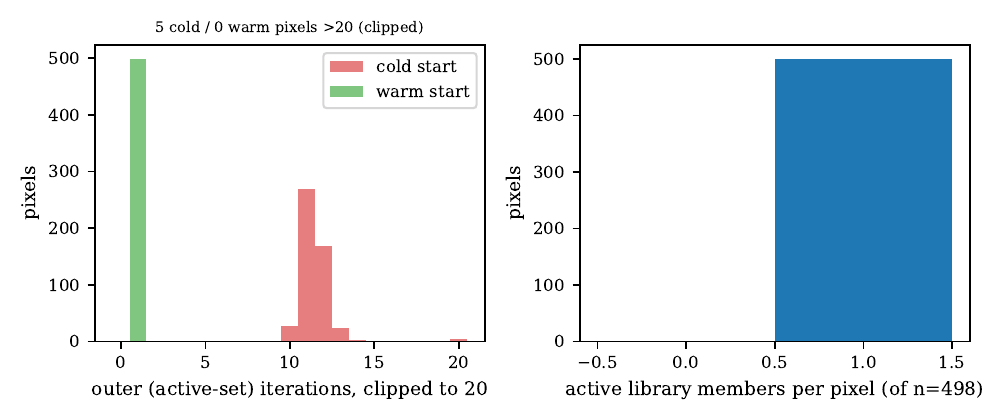}
\caption{Library-scale FCLS unmixing of the Cuprite scene against the USGS spectral library
($n = \nncgHsLibSignatures{}$ signatures, $\nncgHsLibBands{}$ bands, $\nncgHsLibPixels{}$
pixels). \emph{Left}: per-pixel active-set outer-iteration counts, cold versus warm-started;
warm starting collapses them, with $\nncgHsLibFracSingle{}\%$ of pixels solved in a single
outer step. \emph{Right}: active library members per pixel, the loop certifies a mean of
$\nncgHsLibSupportMean{}$, the exact single-mineral support the sparse-unmixing model expects.}
\label{fig:nncg_hyperspectral_library}
\end{figure}

%% file: tables/nncg_synthetic.tex
\begin{tabular}{rrrr}
\toprule
$\kappa$ & Outer steps $s$ & CG inner (total) & Fallback pivots \\
\midrule
$10^{1}$ & 3.0 & 85 & 0.00 \\
$10^{2}$ & 3.2 & 212 & 0.00 \\
$10^{3}$ & 4.2 & 541 & 0.00 \\
$10^{4}$ & 5.0 & 1180 & 0.00 \\
$10^{5}$ & 5.2 & 2294 & 0.00 \\
$10^{6}$ & 6.2 & 4529 & 0.00 \\
\bottomrule
\end{tabular}

%% file: tables/nncg_rankdef.tex
\begin{tabular}{rrcrrr}
\toprule
Support $k$ & $\alpha$ & Converged & CG inner & Capped solves & $\lVert x - x^\star\rVert_\infty$ \\
\midrule
50 & 0.00 & 4/5 & 2716 & 5 & $1\cdot10^{-9}$ \\
50 & 0.05 & 5/5 & 190 & 0 & $6\cdot10^{-10}$ \\
50 & 0.20 & 5/5 & 118 & 0 & $4\cdot10^{-10}$ \\
\midrule
150 & 0.00 & 0/5 & -- & 5 & -- \\
150 & 0.05 & 5/5 & 351 & 0 & $1\cdot10^{-9}$ \\
150 & 0.20 & 5/5 & 163 & 0 & $7\cdot10^{-10}$ \\
\bottomrule
\end{tabular}

%% file: tables/nncg_regu.tex
\begin{tabular}{lrcrrrr}
\toprule
$\alpha$ & $\kappa(A_\alpha)$ & certified & outer & CG inner & active & KKT resid. \\
\midrule
$0$ & $2\cdot10^{20}$ & no & 30 & 826 & 109 & $3\cdot10^{1}$ \\
$1\cdot10^{-6}$ & $8\cdot10^{7}$ & yes & 6 & 115 & 12 & $1\cdot10^{-9}$ \\
$1\cdot10^{-4}$ & $8\cdot10^{5}$ & yes & 3 & 48 & 4 & $2\cdot10^{-9}$ \\
$1\cdot10^{-2}$ & $8\cdot10^{3}$ & yes & 1 & 11 & 0 & $5\cdot10^{-10}$ \\
\bottomrule
\end{tabular}

%% file: tables/nncg_frontier.tex
\begin{tabular}{rrrrr}
\toprule
$n$ & dense $\Sigma$ (GB) & cold (s) & warm (s) & speedup \\
\midrule
2000 & 0.03 & 0.80 & 0.03 & 27$\times$ \\
8000 & 0.51 & 2.63 & 0.05 & 52$\times$ \\
20000 & 3.20 & 5.95 & 0.08 & 72$\times$ \\
\bottomrule
\end{tabular}

%% file: tables/nncg_bench.tex
\begin{tabular}{lrrrr}
\toprule
 & \multicolumn{3}{c}{Time (s), $n = 500$} & \\
\cmidrule(lr){2-4}
Method & $\kappa = 10^2$ & $\kappa = 10^4$ & $\kappa = 10^6$ & $\max\lVert x - x^\star\rVert_\infty$ \\
\midrule
Lawson--Hanson (SciPy) & 0.051 & 0.060 & 0.069 & $5\cdot10^{-13}$ \\
Interior point (Clarabel) & 0.079 & 0.090 & 0.107 & $5\cdot10^{-4}$ \\
Active set $+$ direct & 0.004 & 0.004 & 0.005 & $5\cdot10^{-13}$ \\
Active set $+$ CG & 0.003 & 0.014 & 0.066 & $7\cdot10^{-8}$ \\
\bottomrule
\end{tabular}

%% file: tables/nncg_deblur.tex
\begin{tabular}{rrrrrrrr}
\toprule
$N$ & $n = N^2$ & dense $A$ (GB) & LH (s) & Clarabel (s) & OSQP (s) & MPRGP (s) & NNCG (s) \\
\midrule
32 & 1024 & 0.01 & 0.41 & 0.26 & 0.01 & 0.01 & 0.03 \\
48 & 2304 & 0.04 & 4.72 & 1.64 & 0.08 & 0.02 & 0.05 \\
64 & 4096 & 0.13 & 26.48 & 6.63 & 0.26 & 0.03 & 0.09 \\
96 & 9216 & 0.68 & --- & --- & --- & 0.08 & 0.47 \\
128 & 16384 & 2.15 & --- & --- & --- & 0.19 & 0.59 \\
192 & 36864 & 10.87 & --- & --- & --- & 0.71 & 1.33 \\
256 & 65536 & 34.36 & --- & --- & --- & 1.53 & 2.34 \\
\bottomrule
\end{tabular}

%% file: sections/s8_conclusions.tex
\section{Conclusions}

Conjugate gradients solves SPD systems; it does not, unaided, respect $x \geq 0$. This paper
supplies the missing layer. The bound-constrained quadratic $\min_{x\geq0}
\tfrac{1}{2}x^\top A x - b^\top x$, and its least-squares and equality-augmented ($Bx = c$)
variants, reduces on any fixed free set to an unconstrained SPD solve that CG performs
matrix-free. Non-negativity is enforced around it by a primal-dual active-set loop whose asset
toggles are the principal pivots of $\mathrm{LCP}(A,\,{-b})$. Because $A \succ 0$ is a
$P$-matrix, guarding a fast block-pivot path with a least-index Bland fallback yields
unconditional finite termination at the unique global minimiser (Theorem~\ref{thm:termination}),
even under degeneracy. The inner solver retains CG's $O(\sqrt\kappa)$ Krylov rate,
a factor $\sqrt\kappa$ ahead of the loop-free projected-gradient reference method. A
regularising split $A_\alpha = (1-\alpha)A + \alpha R^\top R$ simultaneously compresses
$\kappa$ and secures the $P$-matrix property, so it both accelerates the inner solve and
licenses the termination proof.

The construction is deliberately application-agnostic: it takes an SPD operator (or a
rectangular $M$ with $A = M^\top M$), a right-hand side, and an optional linear equality
system $Bx = c$, and returns the non-negative minimiser. Handling $Bx = c$ directly
distinguishes the method from the bound-constrained feasible-point solvers (MPRGP and its kin),
which reach a linear equality only through an outer augmented-Lagrangian
loop~\cite{dostal2006smalbe}: the long-only portfolio with a budget or factor-neutrality
constraint is a problem they cannot pose but this loop solves as a single instance. The trade
is genuine, and runs the other way too: the method treats only the non-negative orthant
$x \geq 0$, not the general two-sided box $\ell \leq x \leq u$ those methods support, which
would need a mixed-complementarity formulation we do not develop. The guarantees and complexity
recorded here hold for any SPD $A$. A reference implementation is available as the open-source package \texttt{nncg}
(\url{https://github.com/Jebel-Quant/nncg}), whose test suite reproduces the synthetic study
of Section~\ref{sec:results}.

%% file: sections/s9_appendix.tex
\section{Preconditioning experiments}\label{app:precond}

\paragraph{Jacobi.}
Section~\ref{sec:conditioning}'s claim that a preconditioner slots into the loop unchanged is
tested on operators with a well-conditioned core and a badly scaled diagonal,
$A = D^{1/2}(Q\Lambda Q^\top)D^{1/2}$ with $\kappa(Q\Lambda Q^\top) = 10^2$ and the entries of
$D$ spread over four decades. Jacobi-preconditioned CG inside Algorithm~\ref{alg:elim} runs at
the core's conditioning regardless of the spread (\nncgPcgPCG{} inner iterations against
\nncgPcgBase{} for the unscaled operator), while plain CG pays for the scaling in full
($\nncgPcgCG$ iterations at $\kappa(A) \approx \nncgPcgKappa$)---a factor of $\nncgPcgRatio$
recovered by a preconditioner that costs one vector division per iteration.

\paragraph{Randomized Nyström.}
The low-rank-plus-diagonal preconditioner of Section~\ref{ssec:precond} is built in \texttt{nncg}
by randomized Nyström sketching~\cite{halko2011}, $A \approx U\Lambda U^\top + \mu(I - UU^\top)$
for an orthonormal $U \in \mathbb{R}^{n \times r}$, captured eigenvalues $\Lambda$, and a scalar
tail shift $\mu$; this is the constant-diagonal case $d = \mu\mathbf{1}$,
$\Delta = \Lambda - \mu I$ of that preconditioner. It comes in two variants that trade off
exactly where the free set changes. The first, \texttt{Nystrom}, sketches the \emph{current}
free block $A_{\mathcal{F}}$ on every outer step, using a handful of matrix-free products
against $A_{\mathcal{F}}$ plus a QR and a small eigendecomposition. It is thus always built on
the operator the loop is about to solve, at the price of repeating that work whenever
$\mathcal{F}$ changes. The second, \texttt{GlobalNystrom}, sketches the \emph{full} operator $A$
once. It exploits that restricting a rank-$r$ factorisation to a principal submatrix is exact,
$(U\Lambda U^\top + \mu(I-UU^\top))_{\mathcal{F}} = U_{\mathcal{F}}\Lambda U_{\mathcal{F}}^\top +
\mu(I_{\mathcal{F}} - U_{\mathcal{F}} U_{\mathcal{F}}^\top)$ with $U_{\mathcal{F}} =
U[\mathcal{F},:]$. Every later free set then reuses the one global sketch by masking rows rather
than resketching. The masked $U_{\mathcal{F}}$ is no longer orthonormal, so its inverse needs the
general Sherman--Morrison--Woodbury identity rather than \texttt{Nystrom}'s specialised
identity-plus-projection form; it is built once per free set in $O(|\mathcal{F}|\,r + r^3)$ and
applied in $O(|\mathcal{F}|\,r)$ thereafter. \texttt{GlobalNystrom} pays its larger,
whole-operator sketch cost once, then amortises it across every outer step of a solve and across
the support-stable warm-started re-solves of Proposition~\ref{prop:warm}, exactly where a
resketched-per-block \texttt{Nystrom} pays repeatedly for work \texttt{GlobalNystrom} does once.

The sketch-once design is borne out empirically. On the planted spectral-gap family of the
\texttt{nncg} solver-comparison notebook ($\nncgPrecondN{}$ dominant-over-tail eigenvalues,
sketch rank $\nncgPrecondRank{}$, solved against $\nncgPrecondNSolves{}$ independent right-hand
sides sharing one operator and one solver instance per method, so \texttt{GlobalNystrom}'s cache
persists across solves as Proposition~\ref{prop:warm}'s warm-started re-solves would reuse it),
all of \texttt{CG}, \texttt{Jacobi}, \texttt{Nystrom}, and \texttt{GlobalNystrom} reach the
identical solution on every right-hand side (max disagreement $\nncgPrecondMaxError{}$), so only
the cost differs (Figure~\ref{fig:nncg_precond}). \texttt{GlobalNystrom}'s cumulative cost falls
below \texttt{Nystrom}'s by solve $\nncgPrecondCrossover{}$ and finishes
$\nncgPrecondFinalSpeedup{}\times$ ahead ($\nncgPrecondGlobalFinalMs{}$\,ms versus
$\nncgPrecondNystromFinalMs{}$\,ms), confirming that paying the larger whole-operator sketch once
is cheaper than resketching every free block once the operator is reused often enough.

\begin{figure}[ht]
\centering
\includegraphics[width=0.55\textwidth]{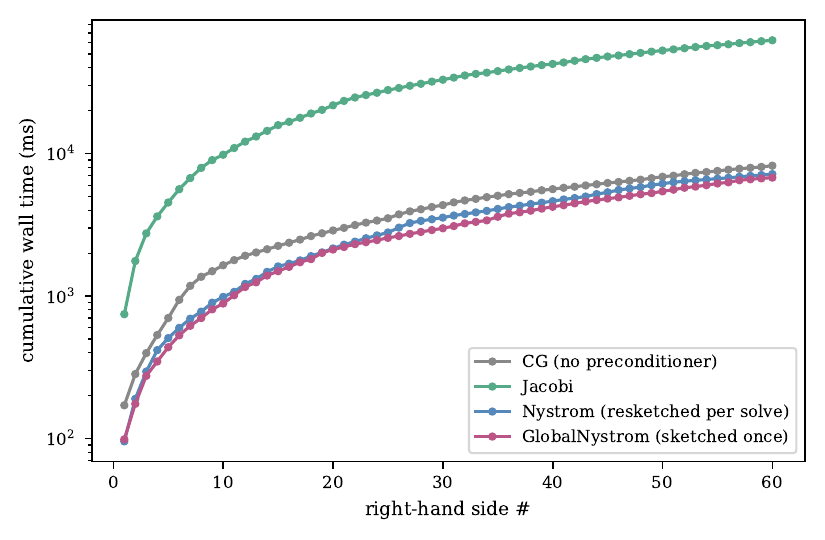}
\caption{Cumulative wall time over $\nncgPrecondNSolves{}$ repeated solves of the same
$n=\nncgPrecondN{}$ spectral-gap operator (sketch rank $\nncgPrecondRank{}$), one line per inner
solver. \texttt{GlobalNystrom} overtakes \texttt{Nystrom} at solve $\nncgPrecondCrossover{}$ and
stays ahead, since its one-time full-operator sketch is amortised across every later free-block
solve rather than repaid on each one.}
\label{fig:nncg_precond}
\end{figure}